\documentclass[12pt]{amsart}
\usepackage[margin=1.19in]{geometry}

\usepackage{amsmath, amscd}
\usepackage{mathrsfs}
\usepackage{amsfonts}
\usepackage{latexsym,epsfig}
\usepackage{mathabx}
\usepackage{amsmath, amscd, amsthm}
\usepackage[pdftex]{color}
\usepackage{comment}
\usepackage{marginnote}
\usepackage[colorlinks,citecolor=blue,linkcolor=red]{hyperref}
\usepackage[nameinlink]{cleveref}
\usepackage{overpic}

\newcommand{\RR}[0]{\mathbb{R}}
\newcommand{\HH}[0]{\mathbb{H}}

\newcommand{\R}{\mathbb{R}}

\renewcommand{\deg}{\mathrm{deg}}

\newcommand{\cir}{\mathrm{circ}}
\newcommand{\til}{\widetilde}
\usepackage{hyperref}

\newtheorem{theorem}{Theorem}[section]
\newtheorem{lemma}[theorem]{Lemma}
\newtheorem{proposition}[theorem]{Proposition}
\newtheorem{corollary}[theorem]{Corollary}
\newtheorem{claim}{Claim}
\newtheorem*{thm}{Theorem}

\theoremstyle{definition}
\newtheorem{remark}[theorem]{Remark}

\newtheorem{question}{Question}

\renewcommand{\S}{\mathcal{S}}

\newcommand{\C}{\mathcal{C}}

\newcommand{\Mod}{\mathrm{Mod}}

\newcommand{\vol}{\mathrm{Vol}}

\newcommand{\T}{\mathbb{T}}
\newcommand{\mc}{\mathcal}

\begin{document}

\title{Covers of surfaces, Kleinian groups, and the curve complex}

\author{Tarik Aougab, Priyam Patel, Samuel J. Taylor}
\address{
\newline Department of Mathematics\\ Brown University
\newline \it{E-mail address}: \tt{tarik\_aougab@brown.edu}
}
\address{
\newline Department of Mathematics\\ University of California, Santa Barbara
\newline \it{E-mail address}: \tt{patel@math.ucsb.edu}
}
\address{
\newline Department of Mathematics\\ Temple University
\newline \it{E-mail address}: \tt{samuel.taylor@temple.edu}
}

\begin{abstract}
We show that curve complex distance is coarsely equal to electric distance in hyperbolic manifolds associated to Kleinian surface groups, up to errors that are polynomial in the complexity of the underlying surface. 
We then use this to control the quasi-isometry constants of maps between curve complexes induced by finite covers of surfaces.
This makes effective previously known results, in the sense that the error terms are explicitly determined, and allows us to give several applications.
In particular, we effectively relate the electric circumference of a fibered manifold to the curve complex translation length of its monodromy, and we give quantitative bounds on virtual specialness for cube complexes dual to curves on surfaces.


\end{abstract}

\maketitle

\section{Introduction}

Let $S$ be an orientable surface of finite type with negative Euler characteristic. The \emph{curve graph} $\C(S)$ of $S$ is the graph whose vertices are homotopy classes of essential simple closed curves and whose edges correspond to pairs of homotopy classes that admit disjoint representatives. A finite-sheeted cover $p \colon \widetilde S \to S$ induces a 
(coarse) map $p^* \colon \C(S) \rightarrow \C(\widetilde S)$ sending a vertex $\gamma$ of $\C(S)$ to its full preimage $p^{-1}(\gamma) \subset \widetilde S$. In \cite{RafiSchleimer}, Rafi--Schleimer show that the map $p^*$ is a $C$--quasi-isometric embedding, with $C$ depending only on $\deg(p)$, the degree of $p$, and on $\chi(S)$. Their result roughly implies that ``pairs of simple closed curves do not detangle very much under pull-back by finite covers of small degree," leading us to pose the following question:

\begin{question}\label{lifts}
Given simple closed curves $\alpha$ and $\beta$ on $S$, what is the minimal degree of a cover $\widetilde S \to S$
for which some components of the preimages $\widetilde \alpha$ and $\widetilde \beta$  are \emph{disjoint}?
\end{question}

Unfortunately, this question cannot be answered using Rafi--Schleimer \cite{RafiSchleimer} as their techniques do not pin down how the constant $C$ depends on $\deg(p)$ and  $\chi(S)$. Therefore, our approach to Question \ref{lifts} is to prove the following theorem:

\begin{thm}[\ref{effective covers}]  \label{effective covers} 
Let $p\colon \widetilde S \rightarrow S$ be a finite covering map between non-sporadic surfaces. 
Then for any $\alpha$ and  $\beta$ distinct essential simple closed curves on $S$, 
\begin{align*}
 \frac{d_{\C(S)}(\alpha, \beta)}{\deg(p) \cdot A_3(|\chi(S)|)}  \leq  d_{\C(\widetilde S)}( p^{\ast}(\alpha), p^{\ast}(\beta)) \leq d_{\C(S)}(\alpha, \beta) , 
 \end{align*}
where $A_3$ is the polynomial $A_3(x) = 80 \; e^{54} \pi x^{13} $ when $S$ is closed. 

When $S$ has punctures, 

\begin{align*}
 \frac{d_{\C(S)}(\alpha, \beta)}{\deg(p)^5 \cdot A_3(|\chi(S)|)}  \leq  d_{\C(\widetilde S)}( p^{\ast}(\alpha), p^{\ast}(\beta)) \leq d_{\C(S)}(\alpha, \beta) , 
 \end{align*}
 where $A_3$ is the polynomial $A_3(x) = 416 \; e^{117}  \pi^{37} x^{38}$.
\end{thm}

\noindent Moreover, the linear dependence on $\deg(p)$ in the first bound of Theorem \ref{effective covers} is sharp. See Remark \ref{Example} for more details.

The polynomial $A_3$ is a product $A_1 \cdot  A_2$ of polynomials 
\begin{align} \label{intro:polys}
A_{1}(x) &=  \begin{cases} 
20 \;  e^{44} x^{10} & \text{for $S$ closed}, \\ \nonumber
 104 \;  e^{94}  \pi^{30}  x^{30} & \text{otherwise}
 \end{cases}
  \\
  \text{and}  
  \\
  \quad A_{2}(x) &=  \begin{cases}  \nonumber
4 \; e^{10}  \pi x^3 & \text{for $S$ closed}, \\
 4 \; e^{23} \pi^7 x^8 & \text{otherwise},
 \end{cases}
\end{align}
which arise independently of one another (e.g., see Theorem \ref{th:distances} below) and we will refer to them often in what follows.

The main ingredient in proving Theorem \ref{effective covers} is the following theorem regarding the relationship between curve graph distance and electric distance in hyperbolic $3$-manifolds. 
Throughout the paper, we use the same notation for a simple closed curve, its corresponding vertex of the curve graph, and its geodesic representative in a $3$-manifold whenever it is clear from context which is meant. In Section~\ref{sec:lemmas} we define a constant $\epsilon_S >0$ which is bounded from below by a polynomial of degree $2$ in $\frac{1}{|\chi(S)|}$ when $S$ is closed and degree $6$ in $\frac{1}{|\chi(S)|}$ in general. This is used in the following theorem.

\begin{thm}[\ref{th:distances}] \label{th:distances}
Let $\alpha$ and $\beta$ be essential simple closed curves in $S$ and let $M$ be a complete hyperbolic structure on $S \times \R$ without accidental parabolics such that $\ell_M(\alpha), \ell_M(\beta) \le \epsilon_S$. Then 
\[
\frac{1}{A_1(|\chi(S)|)} \cdot d_{\mathcal{C}(S)}(\alpha,\beta) \le d^{\epsilon_S}_M(\alpha,\beta) \le A_2(|\chi(S)|) \cdot d_{\mathcal{C}(S)}(\alpha,\beta),
\]
where the polynomials $A_1$ and $A_2$ are as in Equation (\ref{intro:polys}),
and $d^{\epsilon_S}_M$ is the metric obtained from the hyperbolic metric $d_M$ by electrifying the $\epsilon_S$-thin part of $M$.
\end{thm}

In \cite{Tang}, Tang used the original, non-effective (i.e. where the dependence on $|\chi(S)|$ was not explicit) version of Theorem~\ref{th:distances} to reprove the Rafi-Schleimer result, and we follow his argument to obtain Theorem~\ref{effective covers} from Theorem~\ref{th:distances}. 

The non-effective version of Theorem~\ref{th:distances} is originally due to Bowditch \cite[Theorem 5.4]{Bowditch}. (See also the statement of Theorem 4.1 in  Biringer--Souto \cite{BiringerSouto}.)
As Biringer--Souto state in reference to Theorem~\ref{th:distances}, ``credit should also be given to Yair Minsky, since [the theorem] is implicit in the development of the model manifolds of \cite{Minsky}, and to Brock--Bromberg \cite{BrockBromberg}, who prove a closely related result." However, it is important to note that all of the proofs of these results rely on compactness arguments, which cannot be made effective in a way that is necessary for our applications. Thus, the main contribution of Theorem~\ref{th:distances} is that it gives the explicit relationship between curve complex and electric distance. Instead of relying on compactness arguments using geometric limits, we argue using \emph{1-Lipschitz sweepouts in M} 
(see Theorem~\ref{Canary} in Section~\ref{sec:background}).

Indeed, although there are by now many results relating geometric invariants of hyperbolic manifolds to combinatorial invariants of curves on surfaces, almost none of these can quantify or estimate the exact dependence on the complexity of the underlying surface. Moreover, even when such dependence has been estimated, it is usually (at least) exponential in $|\chi(S)|$. For example, Brock's theorems relating volumes of hyperbolic manifolds to distances in the pants graph \cite{Brock,Brock2} are prime examples of important results relating geometry to combinatorics where dependence on the surface was left completely undetermined. However, in forthcoming work of the first and third author with Webb \cite{ATW}, this dependence is bounded using a careful analysis of Masur--Minsky hierarchies \cite{MM2}, but the bound produced is on the order of $|\chi(S)|^{|\chi(S)|}$. Hence, one major novelty of Theorem \ref{th:distances} is that our error terms are explicit and depend polynomially on $|\chi(S)|$. To our knowledge, the only other such results are due to Futer--Schleimer \cite{futer2014cusp} who estimate the cusp area of a hyperbolic manifold in terms of translation length in the arc complex.

\medskip

Using Theorem~\ref{effective covers} we address Question~\ref{lifts} by giving a lower bound on $\text{deg}(\alpha, \beta)$, the minimal degree of a cover necessary to have disjoint components $\widetilde \alpha$, $\widetilde \beta$ of the preimages of $\alpha$, $\beta$, respectively. We emphasize again that this application requires the effective statement of Theorem~\ref{effective covers} proven here.

\begin{corollary}\label{degree}
For two simple closed curves $\alpha$ and $\beta$ on a surface $S$, $$  \frac{d_{\C(S)}(\alpha, \beta)}{C(|\chi(S)|)}  \leq \deg(\alpha, \beta)^a,$$
where $C(x) = 3 A_1(x)  A_2(x)$ with $A_1$ and $A_2$ as in Equation (\ref{intro:polys}). When $S$ is closed, $a=1$, and in general $a=5$.
\end{corollary} 

In Section~\ref{sec:cubes}, we provide an application of this corollary to certain virtually special cube complexes. Given a sufficiently complicated collection $\Gamma$ of curves on a closed surface $S$, Sageev's construction \cite{Sageev} gives rise to a dual CAT(0) cube complex on which $\pi_1(S)$ acts freely and properly discontinuously. The quotient of the complex under this action is a non-positively curved cube complex $\mathfrak{C}_\Gamma$. It is well known by the work of Haglund--Wise \cite{HaglundWise} that $\mathfrak{C}_\Gamma$ is \emph{virtually special}, meaning that it has a finite degree cover whose fundamental group embeds nicely into a right-angled Artin group. The following theorem quantifies this statement by estimating the degree of the required cover in the case that $\Gamma$ is a pair of curves:

\begin{thm}[\ref{virtually special}]\label{virtually special}
Suppose that $\alpha$ and $\beta$ are two simple closed curves that together fill a closed surface $S$. Let $\deg\mathfrak{C}_{\Gamma}$ be the minimal degree of a special cover of the dual cube complex $\mathfrak{C} _{\Gamma}$ to the curve system $\Gamma = \alpha \cup \beta$. Then 
\[
 \frac{d_{\C(S)}(\alpha, \beta)}{C(|\chi(S)|)}  \leq \deg{\mathfrak{C}_{\Gamma}},
 \]
where $C(x) = 3 A_1(x)  A_2(x)$ is a polynomial of degree $13$. 
\end{thm}

Theorem~\ref{virtually special} is related to work of M. Chu \cite{Chu} and J. Deblois, N. Miller, and the second author \cite{DebloisMillerPatel} on quantifying virtual specialness for various hyperbolic manifolds. 

As a second application, we use Theorem~\ref{th:distances} to effectively relate the electric circumference of a fibered manifold $M_\phi$ to the curve graph translation length $\ell_S(\phi)$ of its monodromy $\phi \colon S \to S$. (For definitions, see Section \ref{sec:fibered}.)

\begin{thm}[\ref{thm:circ}]\label{thm:circ}
If $\phi \colon S \to S$ is a pseudo-Anosov homeomorphism of a closed surface $S$, then
\[
 \frac{1}{A_1(|\chi(S)|)}\cdot \ell_S(\phi) \le \cir_{\epsilon_S} (M_\phi) \le A_2(|\chi(S)|) \cdot \big( \ell_S(\phi)+ 2 \big),
\]
where the polynomials $A_1$ and $A_2$ are as in Equation (\ref{intro:polys})
\end{thm}

The outline of the paper is as follows. In Section~\ref{sec:background} we give the necessary background on curve graphs, Margulis tubes in hyperbolic manifolds, pleated surfaces and sweepouts. We then prove various lemmas regarding curves on surfaces and Margulis tubes in 3--manifolds in Section~\ref{sec:lemmas} before proving Theorem~\ref{th:distances} in Sections \ref{sec:electric}, \ref{sec:subsweep}, and \ref{sec:thmproof} and Theorem~\ref{effective covers} in Section~\ref{sec:CCC}. In Section~\ref{sec:cubes} we prove the application regarding cube complexes and in Section~\ref{sec:fibered} we give the application to fibered manifolds.

\section*{Acknowledgements}
The authors would like to thank Darren Long and Dave Futer for helpful discussions. 
We also thank the referee for several suggestions and corrections, including a more direct argument for Lemma \ref{lem:subint}.

The authors acknowledge support from U.S. National Science Foundation grants DMS 1107452, 1107263, 1107367 "RNMS: GEometric structures And Representation varieties" (the GEAR Network). The first author was partially supported by NSF Grants DMS-1502623 and DMS-1807319. The second author was partially supported by NSF Grant DMS--1812014. The third author was partially supported by NSF grants DMS-1744551 and DMS-2102018 and the Sloan Foundation.

\section{Background}\label{sec:background}

Given an orientable surface $S$ with genus $g\ge 0$, $n \geq 0$ punctures, and without boundary,
define its \emph{complexity} as $\omega(S):= 3g+n-4$. 
We call such a surface $S$ \textit{non-sporadic} if $\omega(S)>0$. In order to avoid trivial cases in what follows, we will assume that all surfaces are non-sporadic.

\subsection{Curves on surfaces}

Recall that a simple curve is \emph{essential} if it is neither nullhomotopic nor peripheral (i.e. it does not bound a disk or once punctured disk on $S$). As is usual in the subject, we will generally refer to a vertex $\alpha \in \C(S)^{(0)}$ as a curve. We reserve the term $\emph{loop}$ to refer to an embedded circle in $S$. With these conventions, a curve is represented by a loop. 
 
Given curves $\alpha, \beta \in \C(S)^{(0)}$, their \textit{geometric intersection number}, denoted $i(\alpha, \beta)$, is defined as 
\begin{equation} \label{intersection number}
i(\alpha, \beta) = \min \left\{ | a \cap b| \right\},
\end{equation}
where the minimum is taken over loop representatives $a$ and $b$ of $\alpha$ and $\beta$, respectively.
A surgery argument due to Hempel \cite{Hempel} (see \cite{schleimer2011notes} for the general case) yields the following upper bound on distance in $\mathcal{C}(S)$ in terms of geometric intersection number: 
\begin{equation} \label{Hempel}
 d_{\mathcal{C}(S)}(\alpha, \beta) \leq 2 \log_{2}(i(\alpha, \beta)) + 2.
 \end{equation}
In particular, this shows that $\C(S)$ is connected.

More recently, Bowditch \cite{Bowditch2} proved a stronger version of Equation~(\ref{Hempel}) which is sensitive to the topology of the underlying surface, and which we will need in subsequent sections. We reformulate Corollary $2.2$ of \cite{Bowditch2} as follows: 

\begin{equation} \label{BowditchTM}
d_{\mathcal{C}(S)}(\alpha, \beta) < 2+ 2 \cdot \frac{ \log(i(\alpha, \beta)/2)}{\log((|\chi(S)|-2)/2)},
\end{equation}
so long as the denominator is well-defined and positive, which is the case for all $S$ with $|\chi(S)| \ge 5$.

\subsection{Hyperbolic manifolds and Margulis tubes}
Here we review some basic background on hyperbolic surfaces and $3$-manifolds. 

Let $S$ be a finite area hyperbolic surface and let $p$ be a puncture of $S$. A peripheral curve about $p$ corresponds to a parabolic element of $\pi_1(S)$.
A horodisk in $\mathbb{H}^{2}$ based at a lift $\widetilde p$ of the puncture $p$ will project to a neighborhood of $p$ in $S$. There exists a horocycle $\widetilde Q_{p}$ such that the quotient of $\widetilde Q_{p}$ by $\mbox{stab}(\widetilde p) < \pi_{1}(S)$ is not embedded, and that for any horocycle $\widetilde H$ based at $\widetilde p$ separating $\widetilde Q_{p}$ from $\widetilde p$, $\widetilde H/ \mbox{stab}(\widetilde p)$ is embedded. We call $\widetilde Q_{p}$ the \textit{maximal horocycle} for $\widetilde p$. The open region of $S$ facing $p$ and bounded by the quotient $Q_{p}$ of $\widetilde Q_{p}$ is called the \textit{maximal cusp neighborhood} of $p$.

There is another horocycle $\widetilde H_{p}$ based at $\widetilde{p}$ which projects to an embedded loop of length $2$. The open region bounded between this loop and $p$ is called a \textit{standard cusp neighborhood} of $p$. The standard cusp neighborhood is isometric to the cylinder $(-\infty, \log 2) \times S^{1}$ equipped with the metric 
\begin{equation} \label{metric on cusp}
dx^{2} + e^{2x} d \theta^{2}, 
\end{equation}
where $-\infty \leq x \leq \log(2)$ and $\theta \in S^{1} = [0,1]/(0 \sim 1) $ (see pages 110-112 of \cite{Buser}). 

A key feature of hyperbolic geometry is that the volume of an $n$-dimensional ball of radius $r$, denoted $\mbox{Vol}_{n}(r)$, grows exponentially as a function of $r$. In subsequent sections we will need explicit formulae for this volume in low dimensions, so we record that information here:
\begin{equation} \label{volume of balls}
\mbox{Vol}_{2}(r)= 4\pi\sinh^{2}(r/2), \quad \mbox{Vol}_{3}(r)= \pi(\sinh(2r)-2r). 
\end{equation}
In particular, 
\begin{equation} \label{volume growth}
\mbox{Vol}_{2}(r)= O(e^{r}), \quad \mbox{Vol}_{3}(r)= O(e^{2r}),  
\end{equation}
and the following limits exist: 
\begin{equation} \label{volume small}
\lim_{r \rightarrow 0} \frac{\mbox{Vol}_{2}(r) }{r^{2}}, \frac{\mbox{Vol}_{3}(r)}{r^{3}}. 
\end{equation}

Given a hyperbolic manifold $M$ and $\delta>0$, the \textit{$\delta$-thin part} of $M$, denoted by $M_{(0,\delta)}$, is the set of points in $M$ with injectivity radius less than $\delta/2$. Similarly, the \textit{$\delta$-thick part}, $M_{[\delta, \infty)}$, consists of all points with injectivity radius at least $\delta/2$. Any hyperbolic manifold $M$ is of course the union of its $\delta$-thick and $\delta$-thin parts. 

For $n \geq 2$, there exists a largest $\epsilon_{n}>0$ called the \textit{n-dimensional Margulis constant} so that the $\epsilon_{n}$-thin part of any hyperbolic $n$-manifold $M$ decomposes into a disjoint union of cusps and subsets of the form $\mathbb{T}_{\alpha_{1}}, \mathbb{T}_{\alpha_{2}},..., \mathbb{T}_{\alpha_{n}}$ where $\mathbb{T}_{\alpha_{i}}$ is a tubular neighborhood of the closed geodesic $\alpha_{i}$. 
Thus, the $\epsilon_{2}$-thin part of a hyperbolic surface is homeomorphic to a disjoint union of annuli, and the $\epsilon_{3}$-thin part of a hyperbolic $3$-manifold decomposes as a disjoint union of solid tori and cusps.

 Meyerhoff \cite{Meyerhoff} demonstrated the following lower bound for $\epsilon_{3}$, which we will subsequently need:
\begin{equation} \label{margulis estimate}
 0.104 < \epsilon_3. 
\end{equation}

Given $\delta \leq \epsilon_{3}$, we denote by $\mathbb{T}_{\alpha}(\delta)$  
the component of $M_{(0,\delta)}$ containing the geodesic $\alpha$.
This is called the \textit{Margulis tube} for $\alpha$. We note that $\mathbb{T}_{\alpha}(\delta)$ can be empty if the length of $\alpha$ is greater than $\delta$.  
Furthermore, there is a definite distance between $\mathbb{T}_{\alpha}(\delta)$ and $\partial \mathbb{T}_{\alpha}(\epsilon_{3})$, which goes to infinity as $\delta \rightarrow 0$. A concrete estimate of this growth was obtained recently by Futer--Purcell--Schleimer \cite{FuterPurcellSchleimer}. We will require this estimate in Section \ref{sec:lemmas} and so we record it there in detail. 

\medskip

We now briefly discuss hyperbolic $3$-manifolds $M = \mathbb{H}^3 / \Gamma$ homeomorphic to $S \times \mathbb{R}$, as these manifolds are the focus of this paper. 
Here and throughout, we always consider such a manifold with a fixed homotopy equivalence $\iota \colon S \to M$, called a marking, which allows us to identify homotopy classes of curves in $S$ with those in $M$. Further, we will only consider hyperbolic structures $M$ on $S \times \mathbb{R}$ \emph{without accidental parabolics} meaning that for each essential curve $\alpha$ in $S$, $\iota(\alpha)$ has a unique geodesic representative, whose length we denote by $\ell_M(\alpha)$.


Associated to any such manifold $M$ without accidental parabolics
 is a pair of \textit{end invariants} $(\lambda^{-}, \lambda^{+})$, each of which is either:

\begin{enumerate}
\item (\textit{non-degenerate}) a point in the Teichm{\"u}ller space of $S$, namely a pair $(f, \sigma)$ where $\sigma$ is a complete hyperbolic metric on $S$ and $f: S \rightarrow \sigma$ is a (homotopy class of) homeomorphism;  
\item (\textit{degenerate}) a filling lamination on $S$. 
\end{enumerate}

End invariants describe the behavior of the geometry of $M= \mathbb{H}^{3}/\Gamma$ as one exits out of each of the two topological ends $\mathcal{E}^{-}, \mathcal{E}^{+}$ of $M$. In the non-degenerate case, an end $\mathcal{E}$ of $M$ is foliated by surfaces $S_{t}$ equipped with induced metrics that converge (after a re-scaling) to a hyperbolic metric on $S$. 

 In the degenerate case, Thurston \cite{Thurston} proved that there exists a sequence of simple closed curves on $S$ whose geodesic representatives exit $\mathcal{E}$, and which converge, in the proper sense, to a lamination on $S$. That an end, in general, behaves in exactly one of the above two ways follows from work of Bonahon \cite{Bonahon} and Canary \cite{Canary2}, and ultimately the proof of the Tameness conjecture by Agol \cite{Agol} and Calegari--Gabai \cite{CalegariGabai}.

The celebrated Ending Lamination Theorem, proved by Minsky \cite{Minsky} and Brock-Canary-Minsky \cite{BrockCanaryMinsky}, asserts that the end invariants $(\lambda^{+}, \lambda^{-})$ associated to $\mathcal{E}^{+}, \mathcal{E}^{-}$, respectively, 
determine the hyperbolic manifold $M$.

\subsection{Pleated surfaces and sweepouts.}

Fixing a hyperbolic $3$-manifold $M$, a topological surface $S$, and a lamination $\lambda$ on $S$, a \textit{$\lambda$-pleated surface} is a map $F \colon S \rightarrow M$ so that:
\begin{enumerate}
\item $F$ is proper, and hence sends cusps to cusps,
\item  for each leaf $l$ of $\lambda$, $F(l) \subset M$ is a geodesic,
\item  for each component $R$ of $S \ \setminus \ \lambda$, $F(R)$ is totally geodesic. 
\end{enumerate}

The map $F$ induces via pull-back a complete hyperbolic metric on the surface $S$. With respect to this metric, $F$ is a $1$-Lipschitz map of hyperbolic manifolds. Pleated surfaces arise very naturally in the study of hyperbolic $3$-manifolds. For example, the convex core of a quasi-Fuchsian hyperbolic $3$-manifold is always bounded by the image of a pair of pleated surfaces. Moreover, if $M$ is homeomorphic to $S \times \mathbb{R}$, work of Thurston \cite{Thurston} implies that if all leaves of $\lambda$ can be realized as geodesics in $M$, there exists a $\lambda$-pleated surface into $M$.

Given a surface $S$ and a hyperbolic $3$-manifold $M$, a \textit{1-Lipschitz sweepout} is a homotopy $f_{t}: X_{t} = (S, g_{t}) \rightarrow M$, where $g_{t}$ is a hyperbolic metric on $S$ and $f_{t}$ is a $1$-Lipschitz map for each $t$. 
Whenever $M$ is homeomorphic to $S \times \RR$, we only consider $1$-Lipschitz maps homotopic to our fixed marking.
An important result of Canary \cite{Canary} yields the existence of $1$-Lipschitz sweepouts interpolating between geodesics in $M$.

\begin{theorem} [Canary] \label{Canary} Let $M= \mathbb{H}^{3}/\Gamma$ be a hyperbolic $3$-manifold homeomorphic to $S \times \mathbb{R}$ without accidental parabolics.
Let $\alpha, \beta$ be a pair of essential simple closed curves on $S$ with geodesic representatives $\alpha^{\ast}, \beta^{\ast}$ in $M$. Then there exists a $1$-Lipschitz sweepout $f_{t}\colon X_{t}\rightarrow M, 0\leq t \leq 1$, so that 
\[ \alpha^{\ast} \subset f_{0}(X_{0}) \hspace{ 2 mm} \mbox{and} \hspace{2 mm} \beta^{\ast} \subset f_{1}(X_{1}). \]
\end{theorem}

Theorem \ref{Canary} follows from Canary's work on \textit{simplicial hyperbolic surfaces} \cite{Canary}. These are path-isometric mappings into $M$ of singular hyperbolic surfaces with cone points coinciding with vertices of a geodesic triangulation. More specifically, Canary \cite[Section 5]{Canary} shows the existence of $1$-Lipschitz sweepouts, where $g_{t}$ is a simplicial hyperbolic surface, between any pair of \emph{useful} simplicial hyperbolic surfaces, i.e. ones with a single vertex and a distinguished edge that maps to a closed geodesic. This formulation also appears in \cite[Theorem 4.6]{Brock}.
Brock \cite{Brock} (see also \cite[Theorem 6.7]{BrockBromberg}) reformulates Canary's construction by uniformizing, replacing each $g_{t}$ with the unique non-singular hyperbolic metric in its conformal class and using a result of Ahlfors \cite{ahlfors1938extension} to ensure that the resulting maps are still $1$-Lipschitz. We refer the reader to the proof of Lemma $4.2$ in \cite{Brock} where the complete details are given.

\section{Hyperbolic surfaces and $3$-manifolds}\label{sec:lemmas}
In this section we cover some fairly basic results in hyperbolic geometry. While nothing here will be surprising to experts, care must be taken in order to keep track of how the constants involved depend on the underlying parameters.

For a hyperbolic surface $S$, let $\hat{S}_{\delta}$ denote the compact subset of $S$ obtained by deleting neighborhoods of each cusp consisting of points with injectivity radius at most $\delta/2$.

\begin{lemma} \label{Bears!}
Let $S$ be a non-sporadic surface.
Then for any $\delta \ge 0$ there is a constant $L_{S,\delta} \ge 0$ such that for any finite area hyperbolic structure on $S$ and any $x$ in $\hat{S}_{\delta}$, there is an essential loop in $S$ through $x$ of length less than $L_{S,\delta}$.
\end{lemma}

For the reader's convenience, an upper bound for $L_{S,\delta}$ is recorded in Remark \ref{bottom line}.

\begin{proof}
Let $\widetilde{x}$ be any lift of $x$ to the universal cover $\widetilde{S} = \mathbb{H}^{2}$, and let $\pi : \mathbb{H}^{2} \rightarrow S$ denote the universal covering map. Let $\widetilde{B}$ be a lift of a maximal embedded open ball centered at $x$. By maximality, there must be a pair of points $z,y$ on the boundary of $\widetilde{B}$ which lie in the same fiber over $S$. Then if $[a,b]$ represents the geodesic segment with endpoints $a,b \in \mathbb{H}^{2}$, 
\[ \widetilde{\rho}:= \pi ([\widetilde{x}, z])  \ast \pi([y, \widetilde{x}]) \]
is a loop $\rho$ representing a non-trivial element of $\pi_{1}(S, p)$.

Recall that the area of a hyperbolic disk of radius $r>0$ is $4\pi \sinh^{2}(r/2)$, and therefore by the Gauss-Bonnet theorem the radius of $\widetilde{B}$ is at most
\[ 2 \sinh^{-1} \left(\sqrt{ |\chi(S)|/2} \right). \]
Hence the length of $\widetilde{\rho}$ is at most 
\begin{equation}\label{ell_S} \ell_S := 4 \sinh^{-1} \left( \sqrt{|\chi(S)|/2} \right) = 4 \log \left( \sqrt{|\chi(S)|/2} + \sqrt{1 + |\chi(S)|/2} \right). \end{equation}

When $S$ is closed, this concludes the proof. When $S$ has cusps, the above argument gives us the desired loop unless 
$\rho$ is a peripheral loop that circles a puncture $p$ of $S$. 
In this case, recall that $Q_{p}$ (resp. $H_{p}$) denotes the quotient of a maximal horocycle $\widetilde{Q}_p$ (resp. a standard horocycle $\widetilde{H_{p}}$).

Suppose first that $x$ lies in the standard cusp neighborhood. Since $x \in \hat{S}_{\delta}$, (\ref{metric on cusp}) implies that the distance $d_{S}(x, H_{p})$ between $x$ and $H_{p}$ satisfies 

\begin{equation} \label{distance to standard}
d_{S}(x, H_{p}) \leq \log(2/\delta). 
\end{equation}
 
Let $N_{p}$ be the subset of the maximal cusp neighborhood bounded by $Q_{p}$ and $H_{p}$. Since the area of the neighborhood of a cusp is equal to the length of its boundary, by the Gauss-Bonnet theorem we have that $Q_{p}$ has length at most $2\pi |\chi(S)|$. The region $N_{p}$ can be lifted to a rectangle $\widetilde{N_{p}}$ in the upper half-plane model which is (up to isometry) of the form 
\[ \widetilde{N_{p}} = \left\{ (y,z) \in \mathbb{H}^{2} : 0 \leq y < a, 0< r \leq z \leq b \right\}, \]  
for some positive $a, b$ and $r$.

Then $H_{p}$ lifts to the top edge of $\widetilde{N_{p}}$ and $Q_{p}$ lifts to the bottom edge. Therefore, 
\[ \ell(H_{p})= \frac{a}{b} = 2; \quad \ell(Q_{p})= \frac{a}{r} \leq 2\pi|\chi(S)|. \]
Hence $a=2b$ and so $r \geq b/\pi|\chi(S)|$, and this implies 
\begin{equation} \label{distance between max and standard}
d_S(H_{p}, Q_{p}) \leq \log(\pi |\chi(S)|). 
\end{equation}
 
By maximality, a fundamental domain of $\mathrm{Stab}(\widetilde Q_p) \curvearrowright \widetilde{Q}_p$ will project to a graph $\Gamma$ on $S$ which is not a loop but has a well defined tangent line at each point since the image of $\widetilde Q_p$ in $S$ meets itself tangentially. Moreover, each edge of $\Gamma$ has a well defined normal direction pointing into the cusp corresponding to $p$. Hence, we may pick a point of tangency $v$ (i.e. vertex of $\Gamma$) and consider the immersed subpaths $\rho_1$ and $\rho_2$ each beginning and ending on one of the sides of $v$, whose union is $\Gamma$. Considered as loops of $S$, each $\rho_i$ is  homotopic to a simple closed curve of $S$ and is nontrivial in $\pi_1(S)$. (This last fact follows since neither $\rho_i$ lifts to a loop in $\mathbb{H}^2$). Further, it can not be that both $\rho_i$ are peripheral, since then $S$ would be a $3$-punctured sphere. We conclude that either $\rho_1$ or $\rho_2$ is an essential simple closed curve.

%

Moreover, we may homotope $\rho_i$ to a simple loop of length at most $2 \log(\pi |\chi(S)|) +2$ by
replacing arcs in $\widetilde Q_p$ with the corresponding arcs in $\widetilde H_p$ and using (\ref{distance between max and standard}). In $S$, this amounts to following a geodesic arc to $H_p$ from the basepoint of $\rho_i$, traversing a portion of $H_p$, and then following a geodesic arc back to the basepoint of $\rho_i$. 
Again abusing notation, we refer to these based representatives as $\rho_{1}, \rho_{2}$ and we note also that $\rho_{i}$ is contained completely within $\overline{N_{p}}$ and $\rho_{i}$ meets $H_{p}$. 

Since $x$ lies within distance $\log(2/\delta)$ from $H_p$, it must be within a distance of at most $1 + \log(2/\delta) < 2 + \log(1/\delta)$ from each $\rho_{i}$, and therefore there is an essential loop through $x$ 
of length at most 
\[ 2(2 + \log(1/\delta))+ 2\log(\pi |\chi(S)|)+2 .\]
 If $x \in N_{p}$, it can be at most $\log(\pi|\chi(S)|)+2$ from each $\rho_{i}$ and thus there is an essential loop through $x$ of length at most 
 \[ 2(\log(\pi|\chi(S)|)+2) + 2\log(\pi |\chi(S)|) +2 = 4\log(\pi|\chi(S)|)+ 6. \]
 It remains to consider the case that $x$ is separated from the puncture by $Q_{p}$. Recall the simple loop $\rho$ constructed in the first part of the argument, and that we are assuming that $\rho$ is peripheral. We claim that $\rho$ must meet $Q_{p}$. Indeed, let $\widetilde{p} \in \partial_{\infty}\mathbb{H}^{2}$ denote a lift of the puncture $p$ and let $\widetilde{Q_{p}}$ be the horocycle based at $\widetilde{p}$ projecting to $Q_{p}$. Since $\rho$ is peripheral, there is a lift $\widetilde{\rho}$ bounded by lifts $\widetilde{x_{1}}, \widetilde{x_2}$ of $x$ so that $\widetilde{x_1}$ and $\widetilde{x_2}$ are on the same horocycle $R$ based at $\widetilde{p}$.  By maximality of $\widetilde{Q_p}$, there is another lift $\widetilde{Q_p}'$ of $Q_p$ that is tangent to $\widetilde{Q_p}$ and which intersects $R$ at two points (see Figure~\ref{fig:horo}).
 
Letting $g \in \pi_1(S)$ be the parabolic element corresponding to the peripheral loop $\rho$, all translates of $\widetilde{x_1}$ under the action of $g$ on $\mathbb{H}^2$ lie along $R$ and lie outside of all lifts of the horoball $\overline{Q_p}$ bounded by $Q_p$ since $x$ is separated from $p$ by $Q_p$. Therefore, there exists a lift $\widetilde{\rho}'$ of $\rho$ with endpoints $\widetilde{x_3}$ and $\widetilde {x_4}$ that lie along $R$ such that $\overline{\widetilde{Q_p}} \cup \overline{\widetilde{Q_p}'}$ separate $\widetilde{x_3}$ and $\widetilde {x_4}$ (again see Figure~\ref{fig:horo}). Thus, $\widetilde{\rho}'$ must intersect $\widetilde Q_p \cup \widetilde Q_p '$ so that $\rho$ meets $Q_p$. 

\begin{figure}[h]
\begin{overpic}[trim = 1in 5.5in 2.2in 4in, clip=true, totalheight=0.2\textheight]{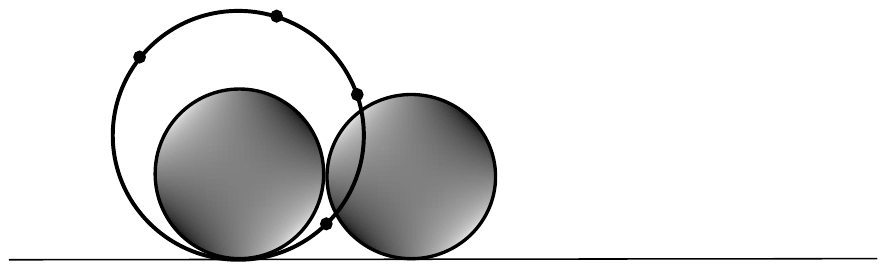}
\put(17,31){\scriptsize{$\widetilde{x_1}$}}
\put(35,36){\scriptsize{$\widetilde{x_2}$}}
\put(42,3){\scriptsize{$\widetilde{x_4}$}}
\put(43,22){\scriptsize{$\widetilde{x_3}$}}
\put(31,26){\scriptsize{$\widetilde{Q_p}$}}
\put(54,25.5){\scriptsize{$\widetilde{Q_p'}$}}
\put(12,15){\scriptsize{$R$}}
\end{overpic}
\caption{When $\rho$ is separated from $p$ by $Q_p$, it must touch $Q_p$.}\label{fig:horo}
\end{figure}

Since the length of $\rho$ is at most $\ell_{S}$, it follows that $x$ must be a distance of at most $\ell_{S}/2$ from $Q_{p}$, and hence from the region $N_{p}$. Thus, there is an essential simple loop through $x$ of length at most  $\ell_{S} + 4 \log(\pi|\chi(S)|)+ 6,$ and so in all three cases the loop constructed has length at most 
\begin{align} \label{def_L_S}
 L_{S, \delta} = 2\log(1/\delta)+ \ell_S + 4\log(\pi|\chi(S)|)+ 6 . 
 \end{align}
\end{proof}

\begin{remark}[Bounds on $L_{S,\delta}$ and $L_S$] \label{bottom line}
We use the proof of Lemma \ref{Bears!} to give an upper bound on $L_{S, \delta}$ that will be useful in subsequent sections. First we use (\ref{ell_S}) to give the following upper bound on $\ell_S$: 
\begin{equation}\label{ell_S bound}
\ell_S \leq 4 \log(2 \sqrt{|\chi(S)|}) = 4\log(2) + 2\log(|\chi(S)|) \leq 3 + 2 \log(|\chi(S)|).
\end{equation}
Note the proof of Lemma \ref{Bears!} is much simpler when $S$ is closed and we can use $\ell_S$ in place of $L_{S, \delta}$ in this case. Combining (\ref{ell_S bound}) and the definition of $L_{S, \delta}$ (Equation \ref{def_L_S}), we have that in general
\begin{equation} \label{simple bound}
 L_{S, \delta} \leq 2\log(1/\delta) + 6\log(\pi|\chi(S)|) + 9. 
 \end{equation}
 
Additionally, when $S$ has punctures, we let $L_S$ denote $L_{S,\epsilon_{3}}$, which using (\ref{margulis estimate}) is at most \[6\log(\pi|\chi(S)|) + 14.\] When $S$ is closed we set $L_S = \ell_S$, which by (\ref{ell_S bound}) is at most \[ 2 \log(|\chi(S)|) +4.\]
\end{remark}

\begin{lemma} \label{lem:prove_short}
Given a non-sporadic surface $S$ there is a constant $0 < \epsilon_S < \epsilon_{3}$ satisfying the following:
Let $M$ be a hyperbolic manifold 
with $M \cong S \times \R$,
and let $\alpha$ be a curve on $S$.
If $f \colon S \to M$ is a $\pi_1$--injective, $1$--Lipschitz map such that $f(S) \cap \mathbb{T}_\alpha (\epsilon_S) \neq \emptyset$, then $\ell_S(\alpha) \le L_{S}$.

Further, there is a loop in the isotopy class of $\alpha$ whose length is less than $L_S$ in $S$ and whose image in $M$ is contained in $\T_\alpha(\epsilon_{3})$.
\end{lemma}

Explicit bounds for the constant $\epsilon_S$ are recorded in Remark \ref{epsilon choice}.

\begin{proof}
Given a positive $\mu < \epsilon_{3}$ and a non-empty $\mu$-tube $\T_{\alpha}(\mu)$, let $\mathcal{F}_\alpha(\mu)$ denote the distance between the boundary of the Margulis tube $\partial\T_{\epsilon_3}(\mu)$ and $\T_{\alpha}(\mu)$. The function $\mathcal{F}_\alpha$ is decreasing in $\mu$, and Theorem $1.1$ of Futer--Purcell--Schleimer \cite{FuterPurcellSchleimer} states that 
\begin{equation} \label{FPS}
\mathcal{F}_\alpha(\mu) \geq \mathcal{F}(\mu) := \mbox{arccosh} \frac{\epsilon_{3}}{\sqrt{7.256 \mu}} - .0424.
\end{equation}

Let $\epsilon_S = \mathcal{F}^{-1}(L_S/2)$ for $L_S$ as in Remark~\ref{bottom line}. Hence (\ref{FPS}) implies 
\[ \mathcal{F}(\epsilon_S)= L_S/2  = \mbox{arccosh} \frac{\epsilon_{3}}{\sqrt{7.256 \epsilon_S}} - .0424, \]
and thus,
\begin{align} \label{eps_S}
 \epsilon_S = \frac{\epsilon_{3}^{2}}{7.256 \cdot \cosh^{2} \left(L_S/2 + .0424 \right)}. 
 \end{align}

If $f(x) \in \mathbb{T}_\alpha (\mu)$ for $\mu < \epsilon_{3}$, then $x \in \hat{S}_{\epsilon_{3}}$ since the $\epsilon_{3}$-thin part of any cusp neighborhood must map via $f$ into the $\epsilon_{3}$-thin part of a cusp neighborhood of $M$, and any Margulis tube in $M$ is disjoint from all $\epsilon_{3}$-thin cusp neighborhoods in $M$. Thus by Lemma \ref{Bears!} there is an essential simple loop $\rho$ through $x$ of length at most $L_S = L_{S,\epsilon_3}$.

Now suppose that $f(x) \in \mathbb{T}_\alpha (\epsilon_S)$ as in the statement of the lemma.
 Since the map is $1$--Lipschitz, the loop $f(\rho)$ has length less than $L_S$ and meets $\mathbb{T}_\alpha (\epsilon_S)$. Hence, any point on $f(\rho)$ has distance at most $L_S/2$ from $\mathbb{T}_\alpha (\epsilon_S)$, and so by construction $f(\rho) \subset \mathbb{T}_\alpha (\epsilon_{3})$.

As $f$ induces an isomorphism on $\pi_1$, $f(\rho)$ is homotopic to some power of $\alpha$. But, $\rho$ is a simple curve on $S$ and so we must have that $\rho$ 
is in the isotopy class of $\alpha$ on $S$ and the proof is complete.
\end{proof}

\begin{remark}[Bounds on $\epsilon_S$]\label{epsilon choice}
Applying the upper bounds on $L_S$ obtained at the end of Remark~\ref{bottom line} and the lower bound on $\epsilon_3$ from (\ref{margulis estimate}) to the definition of $\epsilon_S$ (see Equation \ref{eps_S}), we see that 

 \begin{equation}\label{closed good bound} \epsilon_S \geq \frac{\epsilon_{3}^{2}}{ 8 \cdot \cosh^{2} \left( \log(|\chi(S)|) + 2 \right)} \geq \frac{1}{e^{10}\cdot |\chi(S)|^{2}} \end{equation} when $S$ is closed, and 
\begin{equation}\label{good bound} \epsilon_S \geq \frac{\epsilon_{3}^{2}}{ 8 \cdot \cosh^{2} \left( 3\log(\pi|\chi(S)|) + 7 \right)} \geq \frac{1}{e^{20} \pi^{6} \cdot |\chi(S)|^{6}}, 
\end{equation}
when $S$ has punctures. 

Also, notice in the proof that $\epsilon_S$ is chosen small enough in comparison to $\epsilon_3$ so that the distance between the $\epsilon_S$-tube and the boundary of the $\epsilon_3$-tube is at least $\mathcal{F}(\epsilon_S) = L_S/2\geq 2$. 
\end{remark}

\begin{lemma} \label{lem:bounded_dist}
There is a universal constant $D \ge 0$ so that if $\gamma_1$ and $\gamma_2$ are curves on a hyperbolic surface $S$ with length less than $L_S$, then $d_{\mathcal{C}(S)}(\gamma_1,\gamma_2) \le D$.
\end{lemma}


The proof will show that 
\begin{align} \label{bound_D}
D &\le
\begin{cases} 
  20 & \text{for $S$ closed}, \\
 104 & \text{otherwise}.
 \end{cases}
\end{align}
Following the proof, we give a much smaller bound on $D$, when $|\chi(S)|$ is sufficiently large.

\begin{proof} 
By the collar lemma, $\gamma_1$ has an embedded annular neighborhood of width at least 
\[ \log(\coth(\ell(\gamma_1))/4)=: c(\gamma_1). \]
Since $\log(\coth(x/4)) \rightarrow \infty $ as $x \rightarrow 0$ and decays to $0$ as $x \rightarrow \infty$, there is some positive constant $c$ so that 
\[ c/2 = \log(\coth(c/4)). \]
By inspection we see that $c < 2$. 

We first present a proof in the special case that $S$ is closed, as in this setting the argument is more conceptual.
\vspace{3 mm}

\noindent \textbf{$S$ is closed:} First assume that $\gamma_1$ is the shortest closed geodesic on $S$, and let $\mathcal{N}(\gamma_1)$ denote a maximal collar neighborhood of $\gamma_1$. Let $\widetilde{\mathcal{B}}$ denote 
the boundary of a lift of $\mathcal{N}(\gamma_1)$ to the universal cover.
By maximality, there is a pair of points $\widetilde{x}, \widetilde{y} \in \widetilde{\mathcal{B}}$ which project to the same point on the boundary of $\mathcal{N}(\gamma_1)$ but are not identified by $\mathrm{stab}(\widetilde{\mathcal{B}})$.
Let $\widetilde \gamma_1$ be the lift of $\gamma_1$ corresponding to $\widetilde{\mathcal{B}}$ and
let $\widetilde{x}'$ and $\widetilde{y}'$ be the points on $\widetilde \gamma_1$ nearest to $\widetilde{x}$ and $\widetilde{y}$, respectively. If necessary, we place $\widetilde y$ to minimize the distance between $\widetilde{x}'$ and $\widetilde{y}'$ so that the subarc $a$ of $\widetilde \gamma_1$ between them has length strictly between $0$ and $\ell(\gamma_1) /2$. Hence, the path $\widetilde \rho$ which is the concatenation of the geodesic arcs $[\widetilde x, \widetilde x']$, $a$, and $[\widetilde y', \widetilde y]$ has length at most $\ell(\gamma_1)/2$ plus the width of $\mathcal{N}(\gamma_1)$. 

Let $\rho$ be the loop which is the projection $\widetilde \rho$ to $S$. By construction, $\rho$ is neither homotopically trivial nor homotopic to $\alpha$. 
Since $\gamma_1$ is the shortest closed geodesic in $S$, we conclude that width of $\mathcal{N}(\gamma_1)$  must be at least $\ell(\gamma_1)/2$.

%

Thus, $\gamma_1$ must admit a collar neighborhood of width at least $\ell(\gamma_1)/2$. Therefore, 
\[ i(\gamma_2, \gamma_1) \leq  \min [  2 \ell(\gamma_2)/ \ell(\gamma_1), \ell(\gamma_2)/c(\gamma_1) ] \]
\[ \leq L_S \cdot \min[ 2/\ell(\gamma_1), 1/c(\gamma_1)].\]
Thus, if $\ell(\gamma_1)>2$, $i(\gamma_1, \gamma_2) \leq L_S$, and if $\ell(\gamma_1)<2$, $c(\gamma_1)>1/2$, so we have $i(\gamma_1, \gamma_2) \leq 2 L_S$. 

Consider the case where $|\chi(S)| \geq 5$. Using (\ref{BowditchTM}), we note that if $d_{\C(S)}(\gamma_1, \gamma_2) \geq 6$, then $\gamma_1$ and $\gamma_2$ must intersect at least $\frac{(|\chi(S)|-2)^2}{2}$ times. Using the fact that $L_S \leq 2\log(|\chi(S)|) + 4$ (see Remark \ref{bottom line} when $S$ is closed), it follows that $d_{\C(S)}(\gamma_1, \gamma_2) \leq 5$ so long as $|\chi(S)| \geq  8$ since
\[ x\geq 8  \Rightarrow \frac{(x-2)^2}{2} > 2\cdot(2\log(x)+ 4).\]

 For the finite list of remaining surfaces, we use the fact that on any surface $S$, 
\[ d_{\C(S)}(\alpha, \beta) \leq 2 \log_{2}(i(\alpha, \beta)) + 2. \]
Note that if $|\chi(S)|<8$, $i(\gamma_1, \gamma_2) \leq 2L_S<  16.32$, so we must have 
\[ d_{\C(S)}(\gamma_1, \gamma_2) \leq 10. \]

In general, let $\alpha$ represent the systole of $S$; $\gamma_1$ need not coincide with $\alpha$, but the above argument shows that 
\[ d_{\C(S)}(\gamma_1, \alpha) \leq \begin{cases}
      10 & |\chi(S)|< 8 \\
      5 & |\chi(S)| \ge 8
   \end{cases}
\]
and so by applying the same argument to $\gamma_2$ and then using the triangle inequality,
\[ d_{\C(S)}(\gamma_1, \gamma_2) \le \begin{cases}
      20 & |\chi(S)|< 8 \\
      10 & |\chi(S)| \ge 8.
   \end{cases}
\]

\noindent \textbf{The non-closed case:} 

As for the general case where $S$ is not necessarily closed, we use the collar lemma and argue that $\gamma_1, \gamma_2$ each have embedded collar neighborhoods of width at least 
\[  \log(\coth(L_S/4)),\]
which, applying (\ref{simple bound}), is at least 
\[ \log\left(\coth\left( \frac{3}{2}\log(\pi|\chi(S)|)+\frac{7}{2}\right)\right). \]
It follows that 
\[ i(\gamma_1, \gamma_2) \leq \frac{6\log(\pi|\chi(S)|)+ 14}{\log\left(\coth\left( \frac{3}{2}\log(\pi|\chi(S)|)+\frac{7}{2}\right)\right)} \]

\begin{equation} \label{disgusting}
 = \frac{6 \log( \pi | \chi(S)| ) + 14} { \log \left( \frac{e^{7} \pi^{3} |\chi(S)|^{3} + 1}{e^{7} \pi^{3} |\chi(S)|^{3}-1} \right) }=: W_{S}. 
 \end{equation}
We compute directly that 
\[ W_{S} \leq 2^{40} \cdot |\chi(S)|^{3} \log(|\chi(S)|). \]
Assuming that $|\chi(S)| \ge 10$ and using (\ref{BowditchTM}), we conclude that 
\[ d_{\mathcal{C}(S)}(\gamma_{1}, \gamma_{2}) \leq 2 + 2 \cdot \frac{\log\left( 2^{39} |\chi(S)|^{3} \log(|\chi(S)|)  \right)}{\log\left( (|\chi(S)|-2)/2   \right)    }\]
\[  \leq 2+ \frac{78}{\log(|\chi(S)|-2)-1}+ \frac{8 \cdot \log|\chi(S)|}{\log(|\chi(S)|-2)-1}=: 2+ A+B. \]
As we are assuming that $|\chi(S)|\geq 10$, this is in turn at most 
\[ 2+ 72.3+ 17.06 <92. \] 
On the other hand, if $|\chi(S)|< 10$, $W_{S}$ is bounded from above and applying (\ref{Hempel}) yields
\[ d_{\mathcal{C}(S)}(\gamma_{1}, \gamma_{2}) \leq 104. \]
In conclusion, 
\[ d_{\C(S)}(\gamma_1, \gamma_2) \le \begin{cases}
      104 & |\chi(S)|< 10 \\
      92 & |\chi(S)| \ge 10.
   \end{cases} 
\]
\end{proof}

In reference to the conclusion of the previous proof, 
we note that as $|\chi(S)| \rightarrow \infty$, $A \rightarrow 0, B \rightarrow 8$ and thus for sufficiently large Euler characteristic, we obtain the much smaller bound of $11$. Moreover, using a stronger version of (\ref{BowditchTM}) due to the first author \cite{Aougab}, one can conclude that for all $S$ with $|\chi(S)|$ sufficiently large, 
\[ d_{\C(S)}(\gamma_1, \gamma_2) < 6. \]
However, we will not make use of these improvements here since they do not produce completely explicit constants.

We conclude this section with the following lemma.

\begin{lemma} \label{lem:bounded_short}
Let $0 \le \delta \le 1$ and $L\ge1$. Fix $x\in M_{[\delta,\infty)}$. Then the number of homotopy classes of loops of length less than $L$ based at $x$ is less than
\[
P(L,\delta) := \frac{\vol_{3}(L+\delta)}{\vol_{3}(\delta)}. \]

\end{lemma}

\begin{proof}
The argument is standard, but we provide it for the reader's convenience.

Let $\HH^3 \to M$ be the universal covering and let $\til x$ be a fixed lift of $x$. Let $B'$ be the ball of radius $L$ about $\til x$ so that the based homotopy classes of loops of length less than $L$ at $x$ in $M$ correspond to the translates of $\til x$ in $\HH$ contained in $B'$. Since $x\in M_{[\delta,\infty)}$, the $\delta$--balls about these translates are all disjoint, and since they are contained in the ball $B$ of radius $L+\delta$ about $\til x$, we see that the number of such points is bounded by $\frac{\vol_{3}(L+\delta)}{\vol_{3}(\delta)}$, as required.
\end{proof}

\begin{remark}\label{volume bound} Using (\ref{volume of balls}), (\ref{volume growth}), and (\ref{volume small}), we have that 
\begin{equation} \label{pin down}
P(L, \delta) = \frac{\vol_{3}(L+\delta)}{\vol_{3}(\delta)}= \frac{\sinh(2(L+\delta))- 2(L+\delta)}{\sinh(2\delta)- 2\delta},
\end{equation}
For large $L$ and small $\delta$, 
\begin{equation} \label{big genus}
P(L, \delta) = \frac{\sinh(2(L+\delta))- 2(L+\delta)}{ \left( 2\delta + \frac{1}{6}(2\delta)^{3}+ \frac{1}{120}(2\delta)^{5}+...)  \right) - 2\delta} \leq \frac{\sinh(2(L+\delta))}{\delta^{3}}.
\end{equation} 
\end{remark}

\section{Electric distance}\label{sec:electric}

For a hyperbolic manifold $M$, let $d_M$ denote distance in the hyperbolic metric. Fixing $0<\delta \le \epsilon_3$, let $\check M_\delta$ denote the manifold obtained from $M$ by removing $\delta$-thin cusps. Of course, when $M$ has no cusps, $M = \check M_\delta$. For two points $x,y \in \check M_\delta$, their \textit{$\delta$-electric distance} is defined as
\[
d_M^\delta(x,y) = \inf \{\mathrm{length}(p \cap M_{[\delta, \infty)}) \}
\]
where $p$ varies over all paths with image contained in $\check M_\delta$, joining $x$ and $y$. When $M$ has no cusps, this is the length of the portion of the shortest hyperbolic geodesic joining $x$ and $y$ that occurs outside of the $\delta$-tubes of $M$. Our main technical result is an explicit inequality relating distance in the curve graph of $S$ with electric distance in $M$.

\begin{theorem} \label{th:distances}
Let $\alpha$ and $\beta$ be curves in a non-sporadic surface $S$ and let $M \cong S \times \R$ be a hyperbolic manifold without accidental parabolics such that $\ell_M(\alpha), \ell_M(\beta) \le \epsilon_S$. Then 
\[
1/A_1(|\chi(S)|) \cdot d_{\mathcal{C}(S)}(\alpha,\beta) \le d^{\epsilon_S}_M(\alpha,\beta) \le A_2(|\chi(S)|) \cdot d_{\mathcal{C}(S)}(\alpha,\beta),
\]
where the polynomials $A_1$ and $A_2$ are as in Equation (\ref{intro:polys}).
\end{theorem}

The proof will be completed over the next several sections.
\smallskip

The idea behind the following proposition is simple and well-known to experts. 
\begin{proposition} \label{pleat bounded diameter}
Let $0< \eta \leq \frac{\epsilon_{3}}{e^{6}(\pi |\chi(S)|)^3}$. Then for any curves $\alpha$ and $\beta$ in $S$,
\[
d_M^\eta(\alpha,\beta) \le \frac{4\pi |\chi(S)|}{\eta} \cdot d_{\mathcal{C}(S)}(\alpha,\beta).
\]
\end{proposition}

\begin{remark}\label{closed is better}
When $S$ is closed, $\eta$ need only be less than the Margulis constant $\epsilon_{3}$. This fact will be used in the proof of Theorem~\ref{effective covers}.
\end{remark}

\begin{proof}
Let $\alpha = \alpha_0, \alpha_1, \ldots, \alpha_n = \beta$ be a curve graph geodesic from $\alpha$ to $\beta$. For each $i$, let $f_i \colon X_i= (S, g_i) \to M$ be a pleated surface through $\alpha_i \cup \alpha_{i+1}$. 
In particular, $f_i$ maps each of the geodesic representatives of $\alpha_i$ and $\alpha_{i+1}$ in $X_i$ to 
its geodesic representative in $M$. For clarity, we denote the geodesic representative of $\alpha_i$ in $M$ by $\alpha_i^*$.

For each $i$, pick a point $m_i \in \alpha_i^*$, and let $x_i$ and $y_i$ be points along the geodesic representatives of $\alpha_i$ and $\alpha_{i+1}$ in $X_i$, respectively, so that 
$f_i(x_i) = m_i$ and $f_i(y_i) = m_{i+1}$. Finally, let $p_i$ be the shortest path in $X_i$ from $x_i$ to $y_i$.
The bounded diameter lemma of Thurston and Bonahon 
gives that $\mathrm{length}(p_i \cap (X_i)_{[\eta,\infty)}) \le  \frac{4\pi |\chi(S)|}{\eta}$.
Indeed, following \cite[Lemma 4.5]{Canary}, if $p_i^\eta = p_i \cap (X_i)_{[\eta,\infty)}$, then the $\eta/2$--neighborhood $C_i$ of $p_i^\eta$ is embedded in $X_i$ 
and so

\[\eta/2 \cdot \ell_{X_i}(p^\eta_i) \le \mathrm{Area}(C_i) \le 2\pi |\chi(S)|.\]

Since $f_i$ is $1$-Lipschitz, it maps $\eta$-thin parts of $X_i$ to $\eta$-thin parts of $M$ and so $\mathrm{length}(f_i(p_i) \cap M_{[\eta,\infty)}) \le \ell_{X_i}(p_i^\eta) $. When $M$ has no cusps, this immediately gives that 
$d_M^\eta(m_i, m_{i+1}) \le \ell_{X_i}(p_i^\eta)$.


In the presence of cusps, we argue as follows:
First, we claim that $p_i$ cannot enter any horocyclic cusp neighborhood in $X_i$ whose boundary has length $2/e$. To see this, begin with the standard fact that simple closed geodesics on $X_i$ do not enter any standard cusp neighborhoods.
So for any cusp of $X_i$, the endpoints of $p_i$ lie outside of its standard cusp neighborhood. Since $p_i$ is embedded, the length of any component of its intersection with a standard cusp neighborhood is no more than $2$. Hence, its deepest point in the standard cusp neighborhood has distance no more than $1$ from the horocycle boundary. This means that it does not cross the horocycle for that cusp of length $2/e$.

Now suppose that there is some $z\in p_i$ such that $f_i(z)$ lies in an $\eta$-cusp of $M$. Then any nontrivial loop based at $z$ whose length is less than $2\log(\epsilon_{3}/\eta)$ must be peripheral. This is because the image of such a loop is entirely contained in the corresponding $\epsilon_{3}$-cusp of $M$ and so the loop must represent a peripheral element of $\pi_1 S$. But since  
\[ 
\eta \leq \frac{\epsilon_{3}}{e^{6}(\pi |\chi(S)|)^3}, 
\]
we see that every loop of length no more than $6 \log(\pi |\chi(S)|) +12$ based at $z$ is peripheral. However, the fact that $p_i$ does not enter any horocyclic cusp neighborhood in $X_i$ whose boundary has length $2/e$, together with Lemma~\ref{Bears!} and Equation~(\ref{simple bound}), implies that every point along $p_i$ is the basepoint of some essential (i.e. nonperipheral) loop of length no more than $6 \log(\pi |\chi(S)|) + 11$, a contradiction. Here we are using the fact that the injectivity radius along $p_i$ is at least $1/e$ so we set $\delta = 1/e$ in (\ref{simple bound}).

We conclude that $f_i(p_i)$ does not enter any $\eta$-cups of $M$. Hence, just as in the case without cusps, we conclude that $d_M^\eta(m_i, m_{i+1}) \le \ell_{X_i}(p_i^\eta)$.

Finally, using the fact that $f_i$ maps the $\eta$-thin part of $X_i$ to the $\eta$-thin part of $M$, we obtain
\begin{align*}
d_M^\eta(\alpha,\beta) &\le \sum_{i=0}^{n-1} d_M^\eta(m_i, m_{i+1})  \\
&\le \sum_{i=0}^{n-1} \ell_{X_i}(p_i^\eta)\\
&\le  \frac{4\pi |\chi(S)|}{\eta} \cdot d_{\mathcal{C}(S)}(\alpha,\beta). \qedhere
\end{align*}
\end{proof}

We label the coefficient at the end of the proof above by \begin{equation}\label{coeff}
\mathcal{A}(x, \eta) = \frac{4\pi x}{\eta}.
\end{equation}
Thus, (\ref{coeff}) and (\ref{closed good bound}) yield the inequality
\begin{align}\label{A_2}
\mathcal{A}(|\chi(S)|, \epsilon_S) \leq A_2(|\chi(S)|),
\end{align}
for 

\begin{equation}
A_{2}(x) =  \begin{cases} 
4 \; e^{10}  \pi x^3 & \text{for $S$ closed} \\
 4 \; e^{23} \pi^7 x^8 & \text{otherwise},
 \end{cases}
\end{equation}
as in Equation (\ref{intro:polys}). We note that in the non-closed case $\mathcal{A}(x, \epsilon_S) \leq 4 \; e^{20} \pi^7 x^7$, and that the discrepancy between this polynomial and $A_2(x)$ arises from the retract Lipschitz constant from Lemma \ref{Tang retract}. 

Hence, we can complete the proof of the upper bound in Theorem~\ref{th:distances} using Proposition~\ref{pleat bounded diameter} with $\eta = \epsilon_S$ so that 
\begin{align}
d_M^{\epsilon_S}(\alpha, \beta) &\leq \mathcal{A}(|\chi(S)|, \epsilon_S) \cdot d_{\mathcal{C}(S)}(\alpha,\beta)\\
& \leq A_2(|\chi(S)|) \cdot d_{\mathcal{C}(S)}(\alpha,\beta). \nonumber
\end{align}
Note that for this upper bound there is no requirement on the lengths of $\alpha, \beta$.

\medskip 
The main idea for the other direction is contained in the following lemma. Roughly, the lemma says that as long as we can find a sweepout between the geodesics $\alpha$ and $\beta$ which separates $\alpha$ from $\beta$ at all times, then we obtain the desired bound on curve graph distance in terms of electric distance in $M$. The fact that we can find such a sweepout will be proved in the next section.

\begin{lemma} \label{sweep out}
Let $\alpha$ and $\beta$ be curves in a non-sporadic surface $S$ and let $M \cong S \times \R$ be a hyperbolic manifold without accidental parabolics such that $\ell_M(\alpha), \ell_M(\beta) \le \epsilon_S$. 
Let $p$ be a path in $M$ joining $\T_\alpha(\epsilon_S)$ and $\T_\beta(\epsilon_S)$ and suppose that
\begin{enumerate}
\item  $p$ is contained in $M_{[\epsilon_S,\infty)}$, 
\item there is a $1$-Lipschitz sweepout $(f_t \colon X_t = (S,g_t) \to M)_{t\in [0,1]}$ such that $f_t(S) \cap p \neq \emptyset$ for all $t \in [0,1]$, and  
\item $f_0(S) \cap \T_\alpha(\epsilon_S) \neq \emptyset$ and $f_1(S) \cap \T_\beta(\epsilon_S) \neq \emptyset$.
\end{enumerate}
Then
\begin{align*}
d_{\C(S)}(\alpha,\beta)  \le A_1(|\chi(S)|) \cdot \ell_M(p)
\end{align*}
for $A_1(x)$ as in Equation (\ref{intro:polys}).
\end{lemma}

\begin{proof}
Note that by Lemma \ref{lem:prove_short}, $\ell_{X_0}(\alpha)\le L_S$ and $\ell_{X_1}(\beta) \le L_S$. For each curve $\gamma$ on $S$, we consider the following closed subset of $[0,1]$:
\[
I(\gamma) = \{t \in [0,1] : \ell_{X_t}(\gamma_p) \le L_S\}.
\]
Here $\gamma_p$ is the shortest loop over all representatives of $\gamma$ on $S$ with the property that $f_t(\gamma_p)$ passes through the geodesic $p$. By $(2)$ (and Lemma~\ref{Bears!}), these closed subsets cover $[0,1]$. Moreover, by Lemma \ref{lem:bounded_dist}, if $I(\gamma_1) \cap I(\gamma_2) \neq \emptyset$, then $d_{\mathcal{C}(S)}(\gamma_1,\gamma_2) \le D$, where $D$ is the universal constant from Lemma \ref{lem:bounded_dist}. Finally, that same lemma gives that $d_{\mathcal{C}(S)}(\alpha,\gamma) \le D$ whenever $0 \in I(\gamma)$ and $d_{\mathcal{C}(S)}(\beta,\gamma) \le D$ whenever $1\in I(\gamma)$.

We use this information to build a graph $G$ whose vertices are the curves $\gamma$ such that $I(\gamma) \neq \emptyset$ together with $\alpha$ and $\beta$. Two vertices $\gamma_1$ and $\gamma_2$ of $G$ are joined by an edge in $G$ if either $I(\gamma_1) \cap I(\gamma_2) \neq \emptyset$, or $\gamma_1 = \alpha$ and $0 \in I(\gamma_2)$, or $\gamma_1 = \beta$ and $1 \in I(\gamma_2)$. The first paragraph of the proof immediately implies that the graph $G$ is connected and that adjacent curves have curve complex distance at most $D$. 

The proof will be completed by giving a bound on the number of vertices of $G$ in terms of $\ell_M(p)$. For this, first note that every vertex of $G$ can be realized in $M$ as a loop meeting the path $p$ with length no more than $L_S$.

Break $p$ up into $N$ segments $p_1, \ldots p_N$, the first $N-1$ of which have length $1$ and the last of which has length less than $2$ so that $N= \lfloor \ell_{M}(p) \rfloor$. 
(By Remark \ref{epsilon choice}, $\ell_M(p) \ge 4$.)
Let $\mathcal S_i$ be the set curves on $S$ that can be realized in $M$ as loops meeting the segment $p_i$ with length no more than $L_S$. By the above paragraph, the vertices of $G$ are contained in $\bigcup_{i=1}^N \mathcal{S}_i$ and so the number of vertices of $G$ is no more than $\sum_{i=1}^N \#\mathcal S_i$.

By criterion $(1)$, Lemma \ref{lem:bounded_short}, and Remark \ref{volume bound},
\[
\# \mathcal{S}_i \le \frac{\vol_{3}(L_S+ \epsilon_S + 2)}{\vol_{3}(\epsilon_S)} \leq \frac{\sinh(2(L_S+ \epsilon_S +2))}{\epsilon_S^{3}} \leq \frac{e^{2(L_S + \epsilon_S +2)}}{\epsilon_S^3}. \]

Furthermore, 
\begin{align*} e^{2(L_S + \epsilon_S+2)} & \leq e^{2L_S + 6} \leq e^6 \cdot e^{4\log(|\chi(S)|)+8} \\ & =  e^{14} |\chi(S)|^{4},
\end{align*}
when $S$ is closed, and 

\begin{align*} e^{2(L_S + \epsilon_S+2)} & \leq e^6 \cdot e^{12\log(\pi|\chi(S)|)+28}\\
& =  e^{34} \pi^{12}  |\chi(S)|^{12},   
\end{align*}
when $S$ has punctures. 
Along with the lower bounds on $\epsilon_{S}$ established in (\ref{closed good bound}) and (\ref{good bound}) this implies, 
\[ \#\mathcal{S}_{i} \leq 
s: =
\begin{cases}
      e^{44}  |\chi(S)|^{10} & \text{for $S$ closed} \\
      e^{94}  \pi^{30}  |\chi(S)|^{30} & \text{otherwise}.
   \end{cases} \]


Putting everything together, there is a path from $\alpha$ to $\beta$ in $G$ whose length is less than the total number of vertices of $G$. Since each of these vertices is contains in some $\mathcal{S}_i$ and adjacent vertices have distance at most $D$ in $\C(S)$,

\begin{align*}
d_{\mathcal{C}(S)}(\alpha,\beta) &\le D \cdot \sum_{i=1}^N \# \mathcal{S}_i \\
& \leq D \cdot  s  \cdot \ell_M(p).
\end{align*}

By the proof of Lemma \ref{lem:bounded_dist} (see (\ref{bound_D}), $D \leq 20$ when $S$ is closed and $D \leq 104$ in general, so that setting 
\begin{equation} \label{A_1}
A_{1}(x) =  \begin{cases} 
20 \;  e^{44} x^{10} & \text{for $S$ closed} \\
 104 \;  e^{94}  \pi^{30}  x^{30} & \text{otherwise}
 \end{cases}
\end{equation}
 as in Equation (\ref{intro:polys}) gives us
\begin{align*}
d_{\mathcal{C}(S)}(\alpha,\beta) &\le A_1(|\chi(S)|) \cdot \ell_M(p).
\end{align*}
This completes the proof.
\end{proof}

\section{Separating sweepouts}\label{sec:subsweep}
In what follows, let $\T_\alpha$ be shorthand for the tube $\T_\alpha(\epsilon_S)$. Let $f_t\colon S \to M$, $t \in [a,b]$ be a $1$-Lipschitz sweepout and let $\Sigma_t = f_t(S)$. For a given time $t$, we say that $\Sigma_t$ is \emph{to the left} of $\T_\alpha$ if $\T_\alpha$ lies in the component of $M \setminus \Sigma_t$ containing the $\lambda^+$ end of $M$, and $\Sigma_t$ lies \emph{to the right} of $\T_\alpha$ if $\T_\alpha$ lies in the component of $M \setminus \Sigma_t$ containing the $\lambda^-$ end of $M$. We say that  $\Sigma_t$ is \emph{weakly to the left (resp. weakly to the right)} of $\T_\alpha$ if $\Sigma_t$ is to the left (resp. right) of $\T_\alpha$ or $\Sigma_t$ intersects $\T_\alpha$. 

In order to find sweepouts satisfying the conditions of Lemma~\ref{sweep out}, we require the following:

\begin{proposition} \label{prop:sep}
Let $\alpha, \beta$ be intersecting curves on $S$ whose lengths in $M$ are no more than $\epsilon_S$. Let $f_t\colon S \to M$, $t \in [a,b]$ be a $1$-Lipschitz sweepout such that $\Sigma_a$ lies weakly to the left of $\T_\alpha$ and $\T_\beta$
and $\Sigma_b$ lies weakly to the right of $\T_\alpha$ and $\T_\beta$,
where $\Sigma_t = f_t(S)$.   
Then there is a subinterval $[c,d] \subset [a,b]$ such that 
\begin{enumerate}
\item Both $\T_\alpha$ and $\T_\beta$ meet $\Sigma_c \cup \Sigma_d$,
\item Neither $\T_\alpha$ nor $\T_\beta$ meet $\Sigma_t$ for $t\in (c,d)$, and 
\item $\Sigma_t$ separates $\T_\alpha$ from $\T_\beta$ for each $t \in (c,d)$.
\end{enumerate}

\end{proposition}

The proof requires some notation. Let $m_\alpha \subset [a,b]$ be the set of times the sweepout meets $\T_\alpha$:
\[
m_\alpha = \{t \in [a,b] : \Sigma_t \cap \T_\alpha \neq \emptyset \}.
\]
Define $m_\beta$ similarly, and note that $m_\alpha$ and $m_\beta$ are \emph{disjoint} closed subsets of $[a,b]$, since no $1$-Lipschitz map can meet both $\T_\alpha$ and $\T_\beta$. This follows from the fact that if $\Sigma_t$ meets both $\T_\alpha$ and $\T_\beta$, then by Lemma~\ref{lem:prove_short} there are representative loops $a$ and $b$ on $S$ such that $f_t(a) \subset \T_\alpha$ and $f_t(b) \subset \T_\beta$, and so $a$ and $b$ are disjoint. This contradicts the assumption that $\alpha$ and $\beta$ intersect.

The components of $[a,b] \setminus m_\alpha$ are open in the interval $[a,b]$, and each is a subset of one of three disjoint subsets of $[a,b]$, denoted $l_\alpha, r_\alpha , b_\alpha$ and defined as follows. Define $\l_\alpha$ to consist of those times when $\Sigma_t$ is to the left of $\T_\alpha$. Similarly, let $r_\alpha$ be those times for which $\Sigma_t$ lies to the right of $\T_\alpha$, and let $b_t$ be those time when $\T_\alpha$ lies in a bounded component of $M \setminus \Sigma_t$. Since $\Sigma_t$ always separates $M$, $[a,b] \setminus m_\alpha = l_\alpha \cup b_\alpha \cup r_\alpha$. Define $l_\beta, b_\beta, r_\beta$ in the analogous way. We will think of each point in $[a,b]\setminus m_\alpha$ (or $[a,b]\setminus m_\beta)$ 
as being colored by the subset they are in -- each such point gets an $\alpha$ (resp. $\beta$) color.

With this terminology, we claim that the following lemma immediately proves Proposition \ref{prop:sep}.

\begin{lemma}\label{lem:subint}
There is a closed interval $I \subset [a,b]$ whose interior is a component of $[a,b] \setminus (m_\alpha \cup m_\beta)$ such that 
\begin{enumerate}
\item $I$ has one endpoint in $m_\alpha$ and one endpoint in $m_\beta$, and
\item for each $t$ in the interior of $I$, its $\alpha$ color is different from its $\beta$ color.
\end{enumerate}
\end{lemma}

Proposition \ref{prop:sep} follows from the fact that if $t$ gets a different $\alpha$ color and $\beta$ color (the colors being either $l,r,b$) then $\T_\alpha$ and $\T_\beta$ lie in different components of $M \setminus \Sigma_t$.

\smallskip

We now turn to finding the desired subinterval of $[a,b]$. Let us begin by making a few observations. First,  $m_\alpha$ and $m_\beta$ are closed and disjoint, so components of one cannot accumulate onto a component of the other. Hence, if we are at a component of (say) $m_\alpha$ it makes sense to talk about the component of $m_\beta$ immediately after or before it in the time interval. 
Second, outside of the endpoints $a$ and $b$, we only consider components of $m_\alpha \cup m_\beta$ which have nonempty interior. We call such components \emph{thick}. Note that by continuity of the sweepout, the $\alpha$ color (or $\beta$ color) can change only across a thick $m_\alpha$ component and we call such a component \emph{color changing}. More accurately, if two points in $[a,b]$ are not separated by a thick component of $m_\alpha$, then they have the same $\alpha$ color.  Finally, call an interval in $[a,b] \setminus (m_\alpha \cup m_\beta)$ \emph{switching} 
if it has one endpoint in $m_\alpha$ and one endpoint in $m_\beta$. It is clear that a switching interval must exist: otherwise we can construct a sequence of nested intervals $I_0 \supset I_1 \supset \ldots $ each with one endpoint in $m_\alpha$ and one endpoint in $m_\beta$ such that $\cap I_k = \{x \}$. Since we would necessarily have that $x \in m_\alpha \cap m_\beta$, this is a contradiction.

\begin{proof}[Proof of Lemma \ref{lem:subint}] We first prove the lemma under the strengthened assumption that $\Sigma_a$ lies to the left of $\T_\alpha$ and $\T_\beta$ and $\Sigma_b$ lies to the right of $\T_\alpha$ and $\T_\beta$. Note that, in this case, $a \in l_\alpha \cap l_\beta$ and $b \in r_\alpha \cap r_\beta$. 
We assume that all $m_\alpha$ and $m_\beta$ components discussed in this proof are thick. 
There are only finitely many color changing $m_\alpha$ and  $m_\beta$ components. (This is because any point $t$ in the topological boundary of, say, $m_\alpha$ can be colored depending on what side of $\Sigma_t$ the geodesic representative of $\alpha$ lives on, and nearby points of $[a,b] \setminus \mathrm{int}(m_\alpha)$ must have the same color.)
Now choose an interval $[a_0, b_0] \subseteq [a,b]$ that has the smallest number of total color changing $m_\alpha$ and $m_\beta$ components, such that the $\alpha$ color and the $\beta$ color agree at each of $a_0$ and $b_0$, but that these colors are not the same. Note that $[a,b]$ satisfies the criteria by the strengthened assumption, but it is not necessarily the smallest such interval. 

As we sweep from $a_0$ to $b_0$, we will change the $\alpha$ and $\beta$ colors in some order. Here we are using the fact that the $\alpha$ and $\beta$ colors change between $a$ and $b$. Let us suppose, without loss of generality, that the first to change is the $\alpha$ color. As we continue to sweep, the $\alpha$ color can continue to change or remain the same, but eventually we reach the first $m_\beta$ component, which we call $m_b$, for which the $\beta$ color changes. We claim that the closure $[c,d]$ of any switching interval $(c,d)$ between $m_b$ and the color changing $m_\alpha$ component directly preceding it will be the required interval. 
Assume it is not. Then there exists a time $t \in (c,d)$ at which the $\alpha$ and $\beta$ colors must agree, and therefore, must also agree with the $\alpha$ and $\beta$ color at $a_0$ because the $\beta$ color does not change before $m_b$. Then $[t, b_0]$ is an interval satisfying all of the relevant criteria with at least one less color changing $m_\alpha$ component, contradicting the minimality of $[a_0, b_0]$. 

Now if $\Sigma_a$ lies \emph{weakly} to the left of $\T_\alpha, \T_\beta$, but not to the left of both $\T_\alpha, \T_\beta$, then $\Sigma_a$ intersects $\T_\alpha$ or $\T_\beta$, but not both. Assume it intersects $\T_\alpha$. Then, we replace $[a,b]$ with a larger interval $[a',b]$ where $a'< a'' < a$, $[a', a'')$ is colored with $\ell_\alpha$, $[a'', a]$ is added to $m_\alpha$, and $[a', a]$ is colored with $\ell_\beta$. If instead $\Sigma_a$ intersects $\T_\beta$ and lies to the left of $\T_\alpha$, $[a', a'')$ is colored with $\ell_\beta$, $[a'', a]$ is added to $m_\beta$, and $[a', a]$ is colored with $\ell_\alpha$. Note that we are enlarging the interval and extending the coloring but not altering the sweepout itself.

We analogously extend the right side of the interval to obtain $[a', b']$ if $\Sigma_b$ lies \emph{weakly} to the right of $\T_\alpha, \T_\beta$, but not to the right of both $\T_\alpha, \T_\beta$. We run the combinatorial argument above with the new interval $[a', b']$ and note that $a' \in l_\alpha \cap l_\beta$ and $b' \in r_\alpha \cap r_\beta$ as required. Additionally, the desired switching interval $(c, d)$ must lie in the original interval $[a, b]$ that the sweepout is defined on. This follows from the fact  that the first possible color changing interval is $[a'', a]$ so that the definition of the switching interval shows that the smallest possible value for $c$ is $a$. Similarly, the last possible switching interval is $[b, b'']$ so that the largest possible value for $d$ is $b$. Thus, $[c,d]$ is the desired subinterval of $[a,b]$ and the proof is complete.
\end{proof}

Another method for proving Lemma \ref{prop:sep} was suggested to the authors by Dave Futer. In short, one uses a result of Otal \cite{otal_95, otal_2003}, 
which guarantees that short curves in $M$ are unlinked, to topologically order the short $\gamma_i$ and the $1$-Lipschitz surfaces they meet. Rather than attempt to make effective this technique, we chose instead to employ the direct combinatorial argument found above.

\section{Finishing the proof of Theorem \ref{th:distances}}\label{sec:thmproof}

Recall that $A_{2}(|\chi(S)|)$ is obtained by setting $\eta = \epsilon_S$ in (\ref{coeff}), giving us the upper bound in Theorem \ref{th:distances} \[
d_M^{\epsilon_S}(\alpha,\beta) \le \frac{2\pi |\chi(S)|}{\epsilon_S} \cdot d_{\C(S)}(\alpha,\beta) = A_2(|\chi(S)|) \cdot d_{\C(S)}(\alpha, \beta).
\]

For the lower bound, suppose that $\alpha$ and $\beta$ are given and let $p$ be a path in $\check M_{\epsilon_S}$ minimizing the $\epsilon_S$-electric distance between the geodesic representatives of $\alpha$ to $\beta$ in $M$. Here $\check M_{\epsilon_S}$ is the manifold obtained by removing the $\epsilon_S$-thin cusps of $M$, which are disjoint from the $\epsilon_3$ tubes in $M$. (Recall that by assumption $\ell_M(\alpha),\ell_M(\beta) \le \epsilon_S$.)
Let $\S$ be the set of curves $\gamma$ in $S$ such that $p$ meets $\T_\gamma = \T_\gamma(\epsilon_S)$ in $M$ and index $\S = \{\gamma_i\}_{i=1}^N$ according to the order in which these tubes are met by $p$. (Set $\gamma_0 =\alpha$ and $\gamma_{N+1} = \beta$.) Note that this ordering makes sense since these tubes are disjoint, and by Remark \ref{epsilon choice} the distance between two such tubes is at least $4$.

Let $p_i$ be the subarc of $p$ between the last point of $p \cap \T_{\gamma_i}$ and the first point of $p \cap \T_{\gamma_{i+1}}$. Then these subarcs are disjoint and $p \cap M_{[\epsilon_S,\infty)} = \bigcup_i p_i$ since $p$ is contained in $\check M_{\epsilon_S}$ by assumption.

\begin{lemma}\label{lem:weakly}
There exists a $1$-Lipschitz sweepout $f_t\colon S \to M$, $t \in [a,b]$ such that $\Sigma_a = f_a(S)$ lies weakly to the left of $\alpha, \beta, \gamma_i$ and $\Sigma_b = f_b(S)$ lies weakly to the right of $\alpha, \beta, \gamma_i$ for all $1\leq i \leq N$. 
\end{lemma}

\begin{proof}
Let $\mc E$ denote a fixed end of $M$, i.e. $\mc E = \mc E^\pm$. If $\mc E$ is degenerate, then \cite[The Filling Theorem]{Canary} gives a sequence of useful simplicial hyperbolic surfaces that exit $\mc E$. In particular, such surfaces are eventually to the left (right) of any finite set of geodesics when $\mc E = \mc E^-$ (resp. $= \mc E^+$). If $\mc E$ is non-degenerate, then it corresponds to a component of the boundary of the convex core $\mathrm{core}(M)$, which we denote by $\partial_\pm \mathrm{core}(M)$ when $\mc E= \mc E^\pm$. We recall that $\partial_\pm \mathrm{core}(M)$ is itself a pleated surface and the lemma will follow from that fact that it can be uniformly approximated by useful simplicial hyperbolic surfaces, as in the following claim:

\begin{claim}
Let $f\colon S \to M$ be the pleated surface that represents a boundary component of the convex core of $M$. Then there are simplicial hyperbolic surfaces $f_i \colon S \to M$ such that $f_i \to f$ uniformly on compact sets.
\end{claim}

The proof is essentially well-known to experts but doesn't appear to be explicitly written in the literature. Before sketching the details, let us see how it completes the proof. For each end $\mc E^\pm$, let $f_i^\pm$ be a sequence of useful simplicial hyperbolic surfaces that exit that end or accumulate on the associated convex core boundary, depending on whether the end of degenerate or not. Then the $1$-Lipschitz sweepout from $f_i^-$ to $f_i^+$ given by \Cref{Canary} eventually has the desired form since all closed geodesic of $M$ are contained in $\mathrm{core}(M)$.

It remains to prove the claim. The proof is almost exactly the same as the one used to prove \cite[Theorem A.1]{futer2014cusp}. To make the argument as direct as possible, we use notation and refer to references from their Appendix A.

\begin{proof}[Sketch of proof of claim]
The proof follows from two facts. The first is that $f$ is the uniform limit of pleated surfaces $g_i \colon S \to M$ whose pleating locus $L_i$ is maximal (i.e. all complementary regions are ideal triangles) and has the special property that it has a unique closed leave $c_i$ and all other leaves spiral around $c_i$ in a consistent direction (or exit a cusp). 
Indeed, one can first take any sequence of simple closed geodesics $c_i$ whose Hausdorff limit (or limit in the Chabauty topology, in the presence of cusps; see \cite[Definition A.2]{futer2014cusp}) is a lamination containing $L$. See, for example, the proof of \cite[Lemma A.6]{futer2014cusp}. Then complete each $c_i$ to a lamination $L_i$ of the required form; after passing to a subsequence, the $L_i$ limit to a lamination $L' \supset L$. Compactness of the space of pleated surfaces (\cite[Proposition 5.2.18]{canary1986notes}), applied exactly as in \cite[Lemma A.7]{futer2014cusp}, then implies that after passing to a subsequence, the pleated surfaces $g_i$ pleated along $L_i$ converge to $f$ in the compact-open topology.

The second fact is that each $g_i$ as above, is itself a limit of useful simplicial hyperbolic surfaces. Indeed, this fact appears in Thurston's notes (\cite[Section 8.39]{Thurston}) and follows by observing that $L_i$ can be obtained by starting with a triangulation of $S$ with a single vertex along $c_i$, with $c_i$ appearing as an edge of the triangulation, and `spiraling' the other edges around $c_i$ by isotoping the vertex around the geodesic representative of $c_i$. The resulting simplicial hyperbolic surfaces in turn converge to the pleated surface $g_i$. Taken together, we obtain useful simplicial hyperbolic surfaces that limit to the pleated surface $f$, as required.
\end{proof} 
\end{proof}


Next, our analysis breaks into two cases, depending on whether $\gamma_i$ and $\gamma_{i+1}$ intersect as curves on $S$. If not, then $d_{\C(S)}(\gamma_i,\gamma_{i+1})\le 1$ and we can only say that $\ell_M(p_i) \ge 4$ by Remark~\ref{epsilon choice}.

Now suppose that $\gamma_i$ and $\gamma_{i+1}$ are such that $d_{\C(S)}(\gamma_i,\gamma_{i+1}) \ge 2$.
In this case, apply Proposition~\ref{prop:sep} to obtain (up to reversing the time parameter) a sub-sweepout $(f_t \colon X_t = (S,g_t) \to M)_{t\in [a_i,b_i]}$ with the following properties:

\begin{enumerate}
\item $\Sigma_{a_i} \cap \T_{\gamma_i} \neq \emptyset$ and $\Sigma_{b_i} \cap \T_{\gamma_{i+1}} \neq \emptyset$,
\item  $\Sigma_t \cap p_i \neq \emptyset$ for all $t \in [a_i,b_i]$,
\item  $p_i$ is contained in $M_{[\epsilon_S,\infty)}$.
\end{enumerate}

Note that we are using that $\Sigma_t$ cannot meet both $\T_{\gamma_i}$ and $\T_{\gamma_{i+1}}$ and that since $\Sigma_t$ separates $\gamma_i$ from $\gamma_{i+1}$, it must meet $p_i$.
Hence, we may apply Lemma~\ref{sweep out} to conclude that
\[
d_{\C(S)}(\gamma_i,\gamma_{i+1}) \le A_1(|\chi(S)|) \cdot \ell_M(p_i),
\]

and thus, 
\begin{align*}
d_{\C(S)}(\alpha, \beta) &\le \sum_i d_{\C(S)}(\gamma_i,\gamma_{i+1})\\
& \le A_1(|\chi(S)|) \cdot \sum_i\ell_M(p_i)\\
&\le A_1(|\chi(S)|) \cdot d^{\epsilon_S}_M(\alpha,\beta)
\end{align*}
as wanted. This completes the proof of Theorem~\ref{th:distances}.

\section{Covers and the curve complex}\label{sec:CCC}

In this section, we follow Tang \cite{Tang} and apply Theorem~\ref{th:distances} to analyze maps between curve graphs induced by covering maps of surfaces. 

If $p\colon \widetilde S \rightarrow S$ is a covering map, there is a coarsely well-defined map $p^{\ast}\colon \mathcal{C}(S) \rightarrow \mathcal{C}(\widetilde S)$ induced by $p$; given an essential simple closed curve $\gamma$ on $S$, define $p^{\ast}(\gamma)$ to be the full pre-image $p^{-1}(\gamma) \subseteq \widetilde S$. This will be a multi-curve on $\widetilde S$ corresponding to a complete subgraph of $\mathcal{C}(\widetilde S)$. Given $\alpha$ and $\beta$ vertices of $\mathcal{C}(S)$, we can then define the distance in $\mathcal{C}(\widetilde S)$ between $p^{\ast}(\alpha)$ and $p^{\ast}(\beta)$ to be the diameter of their union:

\[ d_{\mathcal{C}(\widetilde S)}(p^{\ast}(\alpha), p^{\ast}(\beta)):= \mbox{diam}( p^{\ast}(\alpha) \cup p^{\ast}(\beta)).  \]

With this setup, we prove the following:
\begin{theorem} \label{effective covers} 
Let $p\colon \widetilde S \rightarrow S$ be a finite covering map between non-sporadic surfaces. Then for any $\alpha, \beta$ distinct essential simple closed curves on $S$, 
\begin{align*}
 \frac{d_{\C(S)}(\alpha, \beta)}{\deg(p) \cdot A_3(|\chi(S)|)}  \leq  d_{\C(\widetilde S)}( p^{\ast}(\alpha), p^{\ast}(\beta)) \leq d_{\C(S)}(\alpha, \beta) , 
 \end{align*}
where $A_3$ is the polynomial $A_3(x) = 80 \; e^{54} \pi x^{13} $ when $S$ is closed. 

When $S$ has punctures, 

\begin{align*}
 \frac{d_{\C(S)}(\alpha, \beta)}{\deg(p)^5 \cdot A_3(|\chi(S)|)}  \leq  d_{\C(\widetilde S)}( p^{\ast}(\alpha), p^{\ast}(\beta)) \leq d_{\C(S)}(\alpha, \beta) , 
 \end{align*}
 where $A_3$ is the polynomial $A_3(x) = 416 \; e^{117}  \pi^{37} x^{38}$.
\end{theorem}

Recall that $A_3(x) = A_1(x) \cdot A_2(x)$ for $A_1,A_2$ as in Equation \ref{intro:polys}.

\begin{proof} Given $\gamma_1, \gamma_2$ disjoint essential simple closed curves on $S$, $p^{\ast}(\gamma_1)$ will be disjoint from $p^{\ast}(\gamma_2)$. Applying this to the vertices along a geodesic from $\alpha$ to $\beta$ proves the upper bound on $d_{\C(\widetilde S)}(p^{\ast}(\alpha), p^{\ast}(\beta))$ in Theorem \ref{effective covers}. 

For the lower bound, we choose a hyperbolic manifold $M \cong S \times \R$ so that $ \ell_{M}(\alpha)$ and $\ell_{M}(\beta)$ are at most $\epsilon_{S}$. Constructing such a manifold is standard; for example, see \cite[Chapter 8]{Kapovich}. Thus, the first inequality of Theorem \ref{th:distances} implies that
\begin{equation} \label{downstairs}
 d_{\C(S)}(\alpha, \beta) \leq A_1(|\chi(S)|) \cdot d^{\epsilon_{S}}_{M}(\alpha, \beta). \end{equation}

The covering map $p$ gives rise to a covering of $3$-manifolds $p^{\ast}M \to M$. Let $p^{\ast}\alpha, p^{\ast}\beta$ also denote the geodesic representatives in $p^{\ast}M$ of the lifts $p^{-1}(\alpha), p^{-1}(\beta)$, respectively, and let $\gamma$ be a path in $p^\ast M$ from any component of $p^\ast \alpha$ to any component of $p^\ast \beta$. Then $\gamma$ maps to a path in $M$ from $\alpha$ to $\beta$.

When there are no cusps, since a covering map is distance non-increasing and sends the thin part into the thin part, it follows that 
\[d^{\epsilon_S}_M(\alpha, \beta) \leq d^{\epsilon_S}_{p^\ast M} (p^\ast \alpha, p^\ast \beta) , \]
where the right hand side is defined to be the minimum electric distance between a tube about any component of $p^\ast \alpha$ and a tube about any component of $p^\ast \beta$. Combining this observation with (\ref{downstairs}) yields 
\begin{equation} \label{halfway up the stairs}
d_{\C(S)}(\alpha, \beta) \leq A_1(|\chi(S)|) \cdot d^{\epsilon_S}_{p^\ast M} (p^\ast \alpha, p^\ast \beta). 
\end{equation}

When $S$ is closed, the upper bound on $\eta$ in Proposition~\ref{pleat bounded diameter} is simply $\epsilon_3$, the Margulis constant. Thus, applying Proposition \ref{pleat bounded diameter} to the right hand side of (\ref{halfway up the stairs}) with $\eta = \epsilon_S$ we obtain 
\[ d_{\C(S)}(\alpha, \beta) \leq A_1(|\chi(S)|) \cdot \mathcal{A}(|\chi(\widetilde S)|, \epsilon_S) \cdot d_{\C(\widetilde S)}(p^\ast \alpha, p^\ast \beta). \]
Recall that $\mathcal{A}(|\chi(\widetilde S)|, \epsilon_S) =  \deg(p) \cdot \mathcal{A}(|\chi(S)|, \epsilon_S)  \le \deg(p) \cdot A_2(|\chi(S)|)$ by (\ref{A_2}), which yields the lower bound in the closed case. 

In order to apply Proposition \ref{pleat bounded diameter} to the right hand side of (\ref{halfway up the stairs}) when $S$ is not closed, it is necessary to choose $$\eta \leq \frac{\epsilon_{3}}{e^{6}(\pi |\chi(\widetilde S)|)^3},$$ and when the degree of $p\colon \widetilde S \to S$ is large we note that $\epsilon_S$ is not small enough to satisfy the above inequality. Hence, we set $\eta = \min \big \{ \frac{1}{e^{20}\pi^{6}|\chi(S)|^{6}}, \frac{0.104}{e^{6}(\pi |\chi(\widetilde S)|)^3} \big \}$, where $\frac{1}{e^{20}\pi^{6}|\chi(S)|^{6}}$ is the lower bound on $\epsilon_S$ from (\ref{good bound}) and $0.104$ is a lower bound on $\epsilon_3$.

Moreover, it is possible that the projection of the $\eta$-electric geodesic in $p^\ast M$ between $p^\ast \alpha$ and $p^ \ast \beta$ is not an $\eta$-electric path in $M$. Indeed, denoting the $\eta$-electric geodesic in $p^\ast M$ by $\gamma$, $p(\gamma)$ might penetrate an $\eta$-thin cusp. On the other hand, $p(\gamma)$ can not penetrate any $\eta/ \deg(p)$-cusps. To address this possibility, we use the following which appears as Lemma $3$ of \cite{Tang}): 

\begin{lemma} \label{Tang retract} Given $\delta < \eta$, let $r$ denote the natural retract from $\check M_{\delta}$ to $\check M_{\eta}$ corresponding to nearest point projection with respect to the $\delta$-electric and $\eta$-electric metrics on the domain and target, respectively. Then $r$ is $K$-Lipschitz, for $K= \sinh(\eta)/ \sinh(\delta)$. 
\end{lemma}

It follows that 

\[ d^{\eta}_{M}(\alpha, \beta) \leq \frac{\sinh(\eta)}{\sinh(\eta/ \deg(p))} \cdot d^{\eta}_{p^\ast M}(p^\ast \alpha, p^\ast \beta). \]

Suppose first that $\deg(p) \geq (0.104)^{1/3} \cdot e^{14/3} \cdot \pi |\chi(S)|$, so that $\eta = 0.104/(e^{6}(\pi|\chi(\widetilde S)|)^{3})$. Using this $\eta$ and that $\deg(p) = \frac{\chi(\tilde S)}{\chi(S)}$, we have

\[ \frac{\sinh(\eta)}{\sinh(\eta/ \deg(p))} = \frac{ \sinh\left( \frac{0.104}{e^{6}\pi^3|\chi(\widetilde S)|^{3}} \right)}{\sinh \left( \frac{0.104 \cdot |\chi(S)|}{e^{6}\pi^3|\chi(\widetilde S)|^{4}} \right)}. \] 

Since the hyperbolic sine function is monotonically increasing, we can take this quantity to be at most 
\[ \frac{\sinh \left( \frac{0.104}{e^{6} \pi^{3} |\chi(\widetilde S)|^{3}} \right)}{ \sinh \left( \frac{0.104}{e^{6} \pi^{3} |\chi(\widetilde S)|^{4}} \right)}=: f(|\chi(\tilde{S})|). \]
Let $g(x)= f(1/x)$, and note that given $C,D>0$, if $x \cdot |g(x)| < C$ on $(0, D]$, then $|f(x)|< C \cdot x$ on $[1/D, \infty)$.

Plotting $x \cdot g(x)$ with any standard computer algebra system reveals that on $(0, 1/2]$, 
\[ x \cdot |g(x)| < 2. \]
 Since the absolute value of the Euler characteristic of a surface covering a hyperbolic surface is at least $2$, we therefore obtain the bound
\[  d^{\eta}_{M}(\alpha, \beta) \leq 2|\chi(\widetilde S)| \cdot d^{\eta}_{p^\ast M}(p^\ast \alpha, p^\ast \beta). \]


Since $\eta \leq \epsilon_S$,

\[d^{\epsilon_S}_{M} (\alpha, \beta) \leq d^{\eta}_{M} (\alpha, \beta),\] and given the inequality above we have that 

\[d^{\epsilon_S}_{M} (\alpha, \beta) \leq d^{\eta}_{M} (\alpha, \beta) \leq 2|\chi(\widetilde S)| \cdot d^{\eta}_{p^\ast M}(p^\ast \alpha, p^\ast \beta). \]

Now starting with (\ref{downstairs}) and applying Proposition \ref{pleat bounded diameter} with this $\eta$ yields

\[ d_{\C(S)}(\alpha, \beta) \leq A_1(|\chi(S)|) \cdot 2|\chi(\widetilde S)| \cdot \mathcal{A}(|\chi(\widetilde S)|, \eta) \cdot d_{\C(\widetilde S)}(p^\ast \alpha, p^\ast \beta), \]
where 
\[
2 |\chi(\widetilde S)|  \cdot \mathcal{A}(|\chi(\widetilde S)|, \eta) = \deg(p)^5 \cdot 8e^9\pi^4 (|\chi(S)|)^5,
\]
which is less than $\deg(p)^5 \cdot A_2(|\chi(S)|)$, giving the desired result. 

Finally, in the case where $\deg(p) < (0.104)^{1/3} \cdot e^{14/3} \cdot \pi |\chi(S)|$, so that $\eta= \frac{1}{e^{20}\pi^{6}|\chi(S)|^{6}}$, using the monotonicity of hyperbolic sine we have that
\[ \frac{\sinh(\eta)}{\sinh(\eta/ \deg(p))} \leq  \frac{\sinh \left( \frac{1}{e^{20}\pi^{6}|\chi(S)|^{6}}\right)}{\sinh \left(\frac{1}{(0.104)^{1/3}e^{74/3}\pi^{7}|\chi(S)|^{7}} \right)}=: r(|\chi (S)|). \]
As in the first case, we perform a change of variable and set $u(x) := r(1/x)$. Similarly to the first case we plot $x \cdot u(x)$, and obtain an upper bound of $158$ over the interval $(0,1]$. An upper bound on the above quotient of $158 \cdot |\chi(S)|$ follows. 

 Running the argument from the above paragraph with this in place of $2|\chi(\widetilde S)|$ and using the fact that $2 \leq \deg(p)$, we see that $158 \cdot |\chi(S)| \cdot \mathcal{A}(|\chi(\widetilde S)|, \eta) \leq \deg(p)^5 \cdot A_2(|\chi(S)|)$. This completes the proof. 
\end{proof} 

Corollary \ref{degree} is immediate from Theorem \ref{effective covers} after noting that if $$d_{\mathcal{C}(\widetilde S)}(p^{\ast}\alpha, p^{\ast}\beta)\geq 4,$$ then every lift of $\alpha$ intersects every lift of $\beta$.

\begin{remark} \label{Example} 
We conclude this section by showing that the linear dependence on $\deg(p)$ in Theorem \ref{effective covers} is sharp in the closed case. 

  Let $S = S_g$ $(g\ge2)$ be a fixed closed surface with curves $\alpha$ and $\beta$ such that $\alpha$ is nonseparating, $\beta$ is separating, and $\alpha$ and $\beta$ fill $S$.  Set $I = i(\alpha,\beta)$ and let $\widetilde S_n$ be the $n$-fold cyclic cover of $S$ built as follows: Take $n$ copies of $X = S \setminus \alpha$, $X_0. \ldots X_{n-1}$ and glue them cyclically along their boundaries. That is, if we let $\partial X = \alpha^l \cup \alpha^r$, then we glue $\alpha^r_i$ to $\alpha^l_{i+1}$ for each $i$ mod $n$. Rename the resulting curve $\alpha^r_i = \alpha^l_{i+1}$ by $\widetilde \alpha_i$. These are the preimages of $\alpha$ in $\widetilde{S}_n$. 

We note that $\widetilde S_n$ is the cover corresponding to the kernel of the homomorphism $\phi\colon \pi_1(S) \to \mathbb{Z} / n \mathbb{Z}$ taking a loop to its algebraic intersection number with $\alpha$, mod $n$. Hence, $\beta$ has $n$ lifts to the cover $\widetilde S_n$ and any such lift intersects no more than $I$ of the $\widetilde \alpha_i$. In particular, each lift of $\beta$ lives in $X_i \cup X_{i+1} \cup \ldots \cup X_{i+I}$ for some $0 \le i \le n$.

Now set $f = \tau_\beta^{-1} \tau_\alpha$, which is pseudo-Anosov by Thurston's criterion (\cite{Thurston2}). Hence, there is a $\kappa >0$ (depending only on $S$) such that $d_{\C(S)}(\alpha, f^j(\alpha)) \ge \kappa j$ (\cite{MasurMinsky}, \cite{GadTsai}). Since both $\tau_\alpha$ and $\tau_\beta$ fix the homology class of $\alpha$ (recall that $\beta$ is a separating curve), so does $f$. Thus, $f$ fixes the kernel of $\phi$ and hence lifts to a map $\widetilde f \colon \widetilde S_n \to \widetilde S_n$. Indeed, if we denote by $\widetilde \alpha$ and $\widetilde \beta$ the full preimage of $\alpha$ and $\beta$, one such lift is $\widetilde f = \tau_{\widetilde \beta}^{-1} \tau_{\widetilde \alpha}$, a composition of multitwists. 

But then, we must have 
\[
\widetilde f (\widetilde \alpha_1) = \tau_{\widetilde \beta}^{-1} (\widetilde \alpha_1) = \tau_{\overline \beta}^{-1} (\widetilde \alpha_1), 
\]
where $\overline \beta$ is a multicurve consisting of components of $\widetilde \beta$ that meet $\widetilde \alpha_1$. Hence, from our observation above, $\overline \beta$, and therefore $\widetilde f (\widetilde \alpha_1)$, is
supported in $X_{-I} \cup X_{-I+1} \cup \ldots \cup X_{I}$.
Therefore, so long as $n \ge 2I+1$, we have that $d_{\C(\widetilde S_n)}(\widetilde \alpha_1, \widetilde f( \widetilde \alpha_1)) \le 2$.

This construction can be iterated by choosing $n \ge 2j I +1$, and considering $\widetilde f^j (\widetilde \alpha_1)$ This curve is contained in  $Y_{j}:= X_{-jI} \cup X_{i+1} \cup \ldots \cup X_{jI}$, which is a proper subsurface of $\widetilde S_n$. Indeed the Euler characteristic of $Y_{j}$ is $2j\chi(S)$, which, in absolute value, is strictly less than $|\chi(\widetilde{S_n})|$ under the assumption that $n \geq 2jI+1$. Hence, $d_{\C(\widetilde S_n)}(\widetilde \alpha_1, \widetilde f^j(\widetilde \alpha_1)) \le 2$ for $n \ge 2j I +1$. 
 
 Now since $\widetilde f^j(\widetilde \alpha_1) \subset \widetilde f^j(\widetilde \alpha) = \widetilde{f^j (\alpha)}$, we can set $\gamma_j = f^j (\alpha)$ to see that we have produced curves $\alpha$ and $\gamma_j$ on $S$ that have distance at least $\kappa j$ and a degree $n = 2jI+1$ cover $\widetilde S_{n}$ such that $d_{\C(\widetilde S_n)}(\widetilde \alpha, \widetilde \gamma_j ) \le 2$. Informally, we have untangled curves with a cover whose degree is linear in curve graph distance.
\end{remark} 

\section{Application to quantified virtual specialness}\label{sec:cubes}

In this section we give an application of Theorem \ref{effective covers} to dual cube complexes for collections of curves on closed surfaces and their special covers.

\subsection{Dual cube complexes and Sageev's construction} Given a finite and filling collection $\Gamma$ of simple closed curves on a closed surface $S$, Sageev's construction \cite{Sageev} gives rise to a dual CAT$(0)$ cube complex $\widetilde{\mathfrak C_\Gamma}$, on which $\pi_1 S$ acts freely, properly discontinuously, and cocompactly. The quotient of $\widetilde{\mathfrak{C}_\Gamma}$ by this action is a non-positively curved cube complex $\mathfrak{C}_\Gamma$, which can be thought of as a cubulation of the surface $S$ since $\pi_1 S \cong \pi_1\mathfrak{C}_\Gamma$.

The construction of $\widetilde{\mathfrak C_\Gamma}$ roughly goes as follows.
In the language of Wise \cite{Wise}, the full preimage $\widetilde \Gamma$ of $\Gamma$ in the universal cover $\widetilde S$ of $S$ is a union of \emph{elevations}, which each split $\widetilde S$ into two half-spaces. A \emph{labelling} of $\widetilde \Gamma$ is a choice of half-space for each elevation in $\widetilde \Gamma$, and the admissible labelings form the vertex set for $\widetilde{\mathfrak C_\Gamma}$. (For more details on admissible labellings see \cite{Sageev}.) 
Two labellings are joined by an edge when they differ on the choice of a half-space for exactly one elevation. The unique CAT(0) cube complex defined by this 1-skeleton is $\widetilde{\mathfrak C_\Gamma}$, and there is an intersection preserving identification of the curves in the system $\widetilde \Gamma$ with the hyperplanes of $\widetilde{\mathfrak C_\Gamma}$. The action of $\pi_1 S$ on $\widetilde S$ permutes the elevations, inducing an isometry of $\widetilde{\mathfrak C_\Gamma}$. We note that this construction of cube complexes works in a far more general setting. We summarize Sageev's construction with the following theorem:

\begin{theorem}[Sageev]\label{Sageev}
Suppose $\Gamma$ is a finite, filling collection of curves on $S$. Then the dual cube complex $\widetilde{\mathfrak C_\Gamma}$ is CAT(0) and there is an intersection preserving identification of the curves in $\Gamma$ with the hyperplanes of $\widetilde{\mathfrak C_\Gamma}$. The group $\pi_1S$ acts freely, properly discontinously, and cocompactly on $\widetilde{\mathfrak C_\Gamma}$. 
\end{theorem}

\subsection{Virtual specialness} It is well known that there exists a finite cover $\overline {\mathfrak C_\Gamma}$ of $\mathfrak{C}_\Gamma$ which is special \cite{HaglundWise}. Here $\overline {\mathfrak C_\Gamma}$ is called special because its hyperplanes avoid three key pathologies (self-intersecton, direct self-osculation, and inter-osculation). There is an algebraic characterization of specialness \cite{HaglundWise}: that $\pi_1\overline {\mathfrak C_\Gamma}$ embeds in a particular right-angled Artin group (RAAG). The defining graph of this RAAG is the \emph{crossing graph} of $\overline {\mathfrak C_\Gamma}$. The crossing graph of $\overline {\mathfrak C_\Gamma}$ is the simplicial graph whose vertices are hyperplanes of $\overline {\mathfrak C_\Gamma}$ and whose edges connect distinct, intersecting hyperplanes. Thus, Theorem~\ref{Sageev} implies that the specialness of a cube complex dual to a collection of curves on a surface is determined by the intersection pattern of the underlying curves.

Suppose that $\Gamma$ consists of two simple closed curves, $\alpha$ and $\beta$, that together fill the surface $S$.
Consider a finite-degree covering map $p: \widetilde S \to S$, and as in Section \ref{sec:CCC} let $p^{\ast}: \mathcal{C}(S) \rightarrow \mathcal{C}(\widetilde S)$ be the induced map between their curve complexes.

There is also an induced covering map on the level of dual cube complexes $p_{\ast}: \mathfrak C_{\Gamma'} \to \mathfrak{C}_\Gamma$  where $\mathfrak C_{\Gamma'}$ is the dual complex to the curve system $\Gamma' = p^{-1}(\alpha) \cup p^{-1}(\beta)$ on $\widetilde S$ and is also the cover of $\mathfrak{C}_\Gamma$ corresponding to the subgroup $\pi_1 \widetilde S < \pi_1 S \cong \pi_1 \mathfrak{C}_\Gamma$. We record the following lemma as an obstruction to the specialness of $\mathfrak C_{\Gamma'}$.

\begin{lemma}\label{lem:obstruction}
Suppose that $\alpha$ and $\beta$ are two simple closed curves that together fill a surface $S$, and that $p: \widetilde S \to S$ is a finite degree covering map. If every lift of $\alpha$  to $\widetilde S$ intersects every lift of $\beta$ to $\widetilde S$, then the cover $\mathfrak C_{\Gamma'}$ of $\mathfrak{C}_\Gamma$ corresponding to $\pi_1 \widetilde S < \pi_1 S \cong \pi_1 \mathfrak{C}_\Gamma$ cannot be special. 
\end{lemma}

\begin{proof}
Following \cite{RtR}, $\mathfrak{C}_{\Gamma'}$ is special if and only if it admits a local isometry into the Salvetti complex of a particular right-angled Artin group. Indeed, one considers the \textit{crossing graph} $\mathcal{I}$ of $\mathfrak{C}_{\Gamma'}$, whose vertices correspond to hyperplanes of $\mathfrak{C}_{\Gamma'}$ and whose edges correspond to pairs of hyperplanes that cross. Associated to $\mathcal{I}$ is the right-angled Artin group $R_{\mathcal{I}}$ generated by the vertices, and with commuting relations for each edge. Applying Theorem 4.4 in \cite{RtR} to our context yields that $\mathfrak{C}_{\Gamma'}$ is special only if it admits a local isometry into the Salvetti complex for $R_{\mathcal{I}}$. 

Note that the lifts of $\alpha \cup \beta$ to $\tilde{S}$ will correspond to the hyperplanes of $\mathfrak{C}_{\Gamma'}$, and thus the vertices of the crossing graph $\mathcal{I}$ correspond to these lifts. Since no two lifts of $\alpha$ (resp. $\beta$) can cross in $\tilde{S}$, $\mathcal{I}$ is triangle-free, and thus the Salvetti complex for the associated right-angled Artin group is a square complex. 

Now, if every lift of $\alpha$ intersects every lift of $\beta$, then $\mathcal{I}$ is the join of two sets of non-adjacent vertices. Thus, $R_{\mathcal{I}} = F_n \times F_m$ is the product of two free groups. Since a local isometry induces an injection on the level of fundamental groups, $\pi_{1}(\mathfrak{C}_{\Gamma'})$, a surface group, must embed in $R_{\mathcal{I}}$. However,  a surface group cannot embed in the product of two free groups \cite{BaumRose}.
\end{proof}

Note that if $d_{\mathcal{C}(\widetilde S)}(p^{\ast}\alpha, p^{\ast}\beta)\geq 4$, then every lift of $\alpha$ intersects every lift of $\beta$. Thus, Theorem \ref{effective covers} gives us the following: 

\begin{theorem}\label{virtually special}
Suppose that $\alpha$ and $\beta$ are two simple closed curves that together fill a closed surface $S$. Let $\deg {\mathfrak{C}_{\Gamma}}$ be the minimal degree of a special cover of the dual cube complex $\mathfrak{C}_{\Gamma}$ to the curve system $\Gamma = \alpha \cup \beta$. Then $$ \frac{d_{\C(S)}(\alpha, \beta)}{C(S)}  \leq \deg  \mathfrak{C}_{\Gamma},$$
where $C(S)$ is a polynomial in $|\chi(S)|$ of degree ${13}$. 

\end{theorem}

\begin{proof}
Suppose that $p\colon \widetilde S \to S$ is a finite degree cover of the surface $S$ and that $p_{\ast}\colon \mathfrak C_{\Gamma'} \to \mathfrak{C}_\Gamma$ is the induced cover of cube complexes. Additionally, assume that $\mathfrak C_{\Gamma'}$ is special. Theorem \ref{effective covers} gives us that $$\frac{ d_{\C(S)}(\alpha, \beta)}{\deg(p) \cdot A_1(|\chi(S)|) \cdot A_2(|\chi(S)|)} \leq  d_{\C(\widetilde S)}( p^{\ast}(\alpha), p^{\ast}(\beta)).$$ 
Given that $S$ is closed, $A_1(|\chi(S)|)$ is a polynomial of degree 10 in $|\chi(S)|$ and $A_2(|\chi(S)|)$ is a polynomial of degree 3 in $|\chi(S)|$ (see Equation~(\ref{intro:polys})). Lemma \ref{lem:obstruction} shows that $\mathfrak C_{\Gamma'}$ cannot be special unless $d_{\C(\widetilde S)}( p^{\ast}(\alpha), p^{\ast}(\beta)) \leq 3$. Combining these results and solving for $\deg(p)$ gives 
$$  \frac{d_{\C(S)}(\alpha, \beta)}{C(S)}  \leq \deg(p),$$
where $C(S) = 3  A_1(|\chi(S)|)  A_2(|\chi(S)|)$. 
\end{proof}

\section{The circumference of a fibered manifold}\label{sec:fibered}
The methods developed above generalize to effectively relate the electric circumference of a fibered manifold to the curve graph translation length of its monodromy. The noneffective version of this relation has proven useful, for example, in work of Biringer--Souto on the rank of the fundamental group of such manifolds \cite{BiringerSouto}. As in the previous section, we restrict to the case where $S$ is closed.

Let $\phi \in \Mod(S)$ be pseudo-Anosov and denote its mapping torus by $M_\phi$. For $0<\delta<\epsilon_3$, denote the hyperbolic \textit{circumference} and \textit{$\delta$--electric circumference} of $M_\phi$ by $\cir(M_\phi)$ and $\cir_\delta(M_\phi)$, respectively. That is, $\cir(M_\phi)$ is the minimum geodesic length of a loop in $M$ which is not in the kernel of the associated map $\pi_{1}(M_\phi) \rightarrow \mathbb{Z}$, and similarly $\cir_\delta (M_\phi)$ is the minimum $\delta$-electric length of a loop in $M$ which is not in the kernel the map. Let $\ell_S(\phi)$ be the \textit{stable translation length} of $\phi$ in $\C(S)$; for any curve $\alpha$,

\[ \ell_S(\phi) = \lim_{n \rightarrow \infty} \frac{ d_{\mathcal{C}(S)}(\alpha, \phi^{n} \alpha)}{n}. \]

\begin{theorem}\label{thm:circ}
If $\phi \colon S \to S$ is a pseudo-Anosov homeomorphism of a closed surface $S$, then
\[
 \frac{1}{A_1(|\chi(S)|)}\cdot \ell_S(\phi) \le \cir_{\epsilon_S} (M_\phi) \le A_2(|\chi(S)|) \cdot \big ( \ell_S(\phi)+ 2 \big),
\]
where the polynomials $A_1$ and $A_2$ are as in Equation (\ref{intro:polys}).
\end{theorem}

Our argument follows the outline from Brock in \cite{Brock2}. There, Brock extends his theorem on volumes of quasi-fuchsian manifolds to volumes of hyperbolic mapping tori. Similarly, we deduce Theorem \ref{thm:circ} from the tools we used to prove Theorem \ref{th:distances}.

\begin{proof}
Let $M =M_\phi$ and let $N$ be the infinite cyclic cover of $M$ corresponding to $S$. The inclusion $\iota \colon S \to M$ lifts to a marking $\widetilde \iota \colon S \to N$. Let $\Phi$ denote the (isometric) deck transformation of $N$ such that $\widetilde \iota \circ \phi$ is homotopic to $\Phi \circ \widetilde \iota$.
Following the proof of  \cite[Theorem 1.1]{Brock2} there is a $1$-Lipschitz map $f\colon X = (S,g) \to N$ homotopic to $\widetilde \iota$ and a $1$-Lipschitz sweepout $f_t \colon X_t = (S,g_t) \to N$ from $f_0 = f$ to $f_1 = \Phi \circ f \circ \phi^{-1}$. (The hyperbolic structure $X_1$ on $S$ agrees with that of $X$ under $\phi$, up to isotopy.) 
As in Theorem \ref{Canary}, this sweepout has the property that there is some curve $\alpha$ in $S$ such that the geodesic representative of $\alpha$ in $N$ is in the image of $f$. Hence, the geodesic representative of $\phi(\alpha)$ lies in the image of $f_1$.

Let $H \colon S \times [0,1] \to N$ be the homotopy given by $H(x,t) = f_t(x)$ and set $\Sigma_t$ to be the image of $f_t$.
Finally, fix an embedding $h \colon S \to N$ homotopic to $\widetilde \iota$ which lies to the left of the image of $H$. Note that there is some $n_0 \ge 1$ such that $\Phi^{n_0}h(S)$ lies to the right of the image of $H$.

For $n>0$,  define a function $\mathfrak{s}_n \colon [n,n+1] \to [0,1]$ by $\mathfrak{s}_n(x) = x - n$ and let $H^n \colon S \times [0,n] \to N$ denote the homotopy formed by gluing together 
\[
H, \quad \Phi \circ H \circ (\phi^{-1}\times \mathfrak{s}_1), \quad  \ldots,   \quad \Phi^{n-1} \circ H \circ (\phi^{-(n-1)} \times \mathfrak{s}_{n-1})
\] 
to form a sweepout from $f$ to $\Phi^n \circ f \circ \phi^{-n}$. (Note that $H^n$ is indeed continuous since the functions agree on their overlap.)
Also, extend the definition of $\Sigma_t$ for $t \in [0,n]$ to be the image of $H^n(\cdot,t)$, so that in particular $\Sigma_n = \Phi^n(\Sigma_0)$.
Note that the image of $H^n$ is contained in the compact region $C_n$ between $h(S)$ and $\Phi^{n+n_0}( h(S))$.

To prove the first inequality, let $\rho : [0,l] \rightarrow M$ be the shortest loop in $M$ which realizes $\cir_{\epsilon_S}(M)$. Note that $\rho$ cannot be $\epsilon_{S}-$short itself. Otherwise, since the image of $f$ under the covering $N\to M$ necessarily meets $\rho$, the argument in Lemma \ref{lem:prove_short} would produce an essential loop in $S$ which is mapped into the Margulis tube about $\rho$. This would imply that $\rho$ represents an element of the kernel of $\pi_1(M_\phi) \to \mathbb{Z}$, a contradiction.

 Denote by $\widetilde \rho$ the preimage of $\rho$ in $N$ (joining the ends of $N$) and let $\widetilde \rho_n = \widetilde \rho \cap C_n$. Since $C_n$ is the union of $n+n_0$ fundamental domains of $\Phi$, 
\[
\ell_{N}^{\epsilon_{S}}(\widetilde \rho_n) = (n+n_0) \cdot \ell_{M}^{\epsilon_{S}}(\rho).
\]

By choice of $h(S)$ and $n_0$, each $\Sigma_t$ separates the boundary components of $C_n$ (which are $h(S)$ and $\Phi^{n+n_0}h(S)$) for $t \in [0,n]$. Hence, each such $\Sigma_t$ intersects $\widetilde \rho_n$. Now pick any curve $\beta$ that is $L_S$-short on $X$ and observe that $\phi^n(\beta)$ is $L_S$-short on $X_n = \phi^n X$.
Then, using Proposition~\ref{prop:sep} and Lemma \ref{sweep out} as in the proof of Theorem~\ref{th:distances}, we conclude that 
\begin{align*}
d_{\C(S)}(\beta, \phi^n(\beta)) &\le A_1(|\chi(S)|) \cdot \ell_{\epsilon_S}(\widetilde \rho_n) \\
&\le A_1(|\chi(S)|) \cdot (n+n_0) \cdot \ell_{N}^{\epsilon_{S}}(\rho).
\end{align*}
Hence, diving both sides by $n$ and taking $n\to \infty$ shows that 
\[ \ell_S(\phi) \le A_1(|\chi(S)|) \cdot \ell_{M}^{\epsilon_{S}}(\rho),\]
proving the first inequality.

For the second inequality, let $\xi_{n}$ be the shortest electric geodesic in $N$ joining the geodesic representatives of $\alpha$ and $\phi^{n}(\alpha)$, where $\alpha$ is as above. Apply Proposition \ref{pleat bounded diameter} to these curves to obtain 
\[ \ell_{N}^{\epsilon_{S}}(\xi_{n}) \leq A_2(|\chi(S)|) \cdot d_{\mathcal{C}(S)}(\alpha, \phi^{n}(\alpha)). \]
 Alter $\xi_{n}$ to a new path $ \omega_{n}$ as follows: for $0 < j < n$, choose some $x_{j} \in \xi_{n} \cap \Sigma_j$, and connect $x_{j}$ to $\Phi^{j} \widetilde{\rho}(0) \in \Sigma_j$ by a shortest electric path $\gamma_{j}$ in $\Sigma_j$. For $j=0$ and $j=n$, define $\gamma_{j}$ to be a shortest electric path in $\Sigma_{j}$ starting at the initial and terminal points of $\xi_{n}$ and ending at lifts $x_{0}$ and $x_{n}$ of $\rho(0)$ in $\Sigma_0$ and $\Sigma_n$, respectively. Then define $ \omega_{n}$ to be the path obtained from $\xi_{n}$ by inserting $\gamma_{j} \ast \gamma_{j}^{-1}$ after $x_{j}$ for each $0<j<n$, and by inserting $\gamma_{0}^{-1}$ at the beginning and $\gamma_{n}$ at the end. Using the bounded diameter lemma (as in the proof of Proposition \ref{pleat bounded diameter}), we have that 

\[ \ell_{N}^{\epsilon_{S}} ( \omega_{n}) \leq \ell_{N}^{\epsilon_{S}}(\xi_{n}) + 2n \cdot \frac{4\pi|\chi(S)|}{\epsilon_S} \]
\[ \leq \ell_{N}^{\epsilon_{S}}(\xi_{n}) + 2n \cdot A_{2}(|\chi(S)|). \]

Let $\omega_{n}[j-1,j]$ denote the portion of $\omega_{n}$ between $\Phi^{j-1} \widetilde \rho(0)$ and $\Phi^{j} \widetilde \rho(0)$. Since $ \omega_{n}[j-1,j]$ descends to a loop in $M$ which is not in the kernel of $\pi_{1}(M) \rightarrow \mathbb{Z}$, we have
\[ \ell_{M}^{\epsilon_{S}}(\rho) \leq \ell_{N}^{\epsilon_{S}}(\omega_{n}[j-1,j]) \hspace{3 mm} \forall j, \]
hence  
\begin{align*}
 n \cdot \ell_{M}^{\epsilon_{S}}( \rho) &\leq  \ell_{N}^{\epsilon_{S}} (\xi_{n}) + 2n\cdot A_{2}(|\chi(S)|)\\
 & \leq A_2(|\chi(S)|) \cdot d_{\mathcal{C}(S)}(\alpha, \phi^{n}(\alpha))+ 2 n \cdot  A_{2}(|\chi(S)|). 
\end{align*}
Dividing through by $n$ and taking a limit as $n \rightarrow \infty$ produces the second inequality. 
\end{proof}

\bibliographystyle{alpha}
\bibliography{ThickDistanceBib}

\end{document}